\documentclass[psamsfonts,12pts]{amsart}

\usepackage{mathpazo}
\usepackage[utf8]{inputenc}
\usepackage[centering]{geometry}
\usepackage{amsmath}
\usepackage{amsthm}
\usepackage{amssymb}
\usepackage{amscd}
\usepackage{amsfonts}
\usepackage{amsbsy}
\usepackage{graphicx}
\usepackage[dvips]{psfrag}
\usepackage{array}
\usepackage{color}
\usepackage{epsfig}
\usepackage{url}
\usepackage{overpic}
\usepackage{epstopdf}
\usepackage[normalem]{ulem}
\usepackage{tikz}

\usepackage{caption}
\usepackage{subcaption}
\usepackage{multirow}
\usepackage{indentfirst}
\usepackage{mathrsfs}
\usepackage{hyperref}
\usepackage{placeins}
\usepackage{flafter}
\usepackage{tcolorbox}
\usetikzlibrary{arrows}

\newtheorem {theorem} {Theorem}
\newtheorem {definition} {Definition}
\newtheorem {proposition} [theorem]{Proposition}
\newtheorem {corollary} [theorem]{Corollary}

\allowdisplaybreaks
\begin{document}
	\renewcommand{\arraystretch}{1.5}
	
	\title[Limit cycles for piecewise smooth differential equations separated by $S^1$]{Limit cycles for classes of piecewise smooth differential equations separated by the unit circle}
	\author[M. D. A. Caldas, R. M. Martins]{Mayara D. A. Caldas$^{1}$, Ricardo M. Martins$^{2}$}
	
	\address{$^{1,2}$ Departamento de Matem\'{a}tica, Universidade Estadual de Campinas, Rua S\'{e}rgio Baruque de Holanda, 651, Cidade Universit\'{a}ria Zeferino Vaz, 13083--859, Campinas, SP, Brazil} 
	
	\email{mcaldas@ime.unicamp.br$^{1}$, rmiranda@unicamp.br$^{2}$}
	
	
	
	\maketitle

\begin{abstract}
	In this article we study the existence of limit cycles in families of piecewise smooth differential equations having the unit circle as discontinuity region. We consider families presenting singularities of center or saddle type, visible or invisible, as well as the case without singularities. We establish an upper bound for the number of limit cycles and give examples showing that the maximum number of limit cycles can be reached. We also discuss the existence of homoclinic cycles for such differential equations in the saddle-center case.
\end{abstract}

\section{Introduction}
	
In this paper we study planar piecewise smooth vector fields
\begin{equation}\label{1}
	(\dot x,\dot y)=Z_{XY}(x,y)=\left\{
	\begin{array}{l}
		X(x,y), \ h(x,y)\leq h_0,\\
		Y(x,y), \ h(x,y)\geq h_0,\\
	\end{array}
	\right.
\end{equation}
where $X,Y:U\subset\mathbb R^2\rightarrow \mathbb R^2$ are affine vector fields, $h:U\rightarrow \mathbb R$ is a smooth function having $h_0\in\mathbb R$ as regular value, and $U$ is an open set.
Over the smooth submanifold $\Sigma=h^{-1}(\{h_0\})$, we assume that the dynamics of $Z$ is provided by Filippov's convention (that we describe in details in Section \ref{fili}).
	
There are several notions of solutions for piecewise smooth differential equations (see for instance \cite{jorge}), and Filippov's notion seems to be suitable for both theoretical developments and applications; see for instance \cite{eco} for an application of Filippov's convention to study a macroeconomic model, \cite{livro1} for applications of piecewise smooth differential equations in real-world models. State of the art bifurcations are described in the survey \cite{survey}.
	
The existence of periodic orbits for piecewise smooth differential equations is a very active area of research and was studied by several authors (see for instance \cite{doug2}). Besides the interest in their applications, one of the main motivations from a theoretical point of view comes from an piecewise smooth version of Hilbert's 16th problem (see \cite{roussarie1,jaume2,coloquio,yu} for the original Hilbert's 16th). 
	
In the case of piecewise smooth differential equations separated by a straight line (for instance, taking $h(x,y)=y$ from \eqref{1}) it is very simple to produce examples with one limit cycle; to produce examples with two limit cycles is not so simple, but it can be done by means of ``closing equations''; to produce an example with three limit cycles is not so straightforward, and it was achieved for the first time in \cite{china1} by using numerical methods. The same result was proved using analytical methods in \cite{3ciclos}, indeed proving the existence of the (same) three cycles using the Newton-Kantorovich theorem. Currently it is conjectured that 3 is the maximum number of limit cycles in this case. Note that if the discontinuity is not over a line, then the result is not true (see for instance \cite{lf} and \cite{doug}). Upper bounds for the maximum number of limit cycles for non-linear piecewise smooth differential equations has also been studied extensively (see for instance \cite{aninha}).

More general results about planar piecewise smooth differential equations, for instance, the local structural stability of planar Filippov differential equations were obtained in \cite{GST}, together with a complete classification of generic bifurcations of low codimension.
	
We are interested in obtaining conditions to the existence  of crossing limit cycles in piecewise smooth differential equations $Z_{XY}$, where $X,Y$ are affine vector fields separated by the unitary circle $S^1=\{(x,y)\in\mathbb R^2, \ x^2+y^2=1\}$. Similar differential equations were studied in \cite{LT,JJJ}.

	

\subsection{Main results}
	
Let $U\subset\mathbb R^2$ be an open and convex set with $(0,0)\in U$. Consider $S^1$ a circle centered at the origin and of radius $1$ such that $S^1\subset U$ and consider $h:U\subset\mathbb R^2\rightarrow\mathbb R$ given by $h(x,y)=x^2+y^2-1$, so that $S^1=h^{-1}(\{0\})$. Let $\mathscr{X}^r$ be the set of vector fields of class $C^r$ defined in $U$. 
	
Given $X,Y\in \mathscr{X}^r$, consider the piecewise smooth vector field
\[Z_{XY}(x,y)=\left\{
\begin{array}{lcl}
	X(x,y),& h(x,y)\leq 0,\\
	Y(x,y),& h(x,y)\geq 0,
\end{array}
\right.\]
where $X$ is defined in the region interior to $S^1$ and $Y$ in the region exterior to $S^1$ and we denote by $\mathfrak{X}^r$  the set of vector fields of type $Z_{XY}$, which can be taken as $\mathfrak{X}^r=\mathscr{X}^r\times \mathscr{X}^r$ inheriting a topology from this product. We introduce the following subsets of $\mathfrak{X}^r$:
	
\begin{itemize}
	\item $\mathfrak{X}^{S^1}_0$ is the set of the piecewise smooth dynamic systems $Z_{XY}$ where $X$ is a constant vector field and $Y$ is a linear vector field with the singularity of center type; 
		
	\item $\mathfrak{X}^{S^1}_1$ is the set of the piecewise smooth dynamic systems $Z_{XY}$ with $X$ a constant vector field and $Y$ a linear vector field with the singularity of saddle type and  $div(Y)=0$;
		
	\item $\mathfrak{X}^{S^1}_2$ is the set of the piecewise smooth dynamic systems $Z_{XY}$ with $X$ a linear vector field with singularity of saddle type and $div(X)=0$ and $Y$ a linear vector field with the singularity of center type;
		
	\item $\mathfrak{X}^{S^1}_3$ is the set of the piecewise smooth dynamic systems $Z_{XY}$ with $X$ a linear vector field with singularity of center type and $Y$ a linear vector field with the singularity of saddle type and $div(Y)=0$;
		
	\item $\mathfrak{X}^{S^1}_4$ is the set of the piecewise smooth dynamic systems $Z_{XY}$ with $X$ and $Y$ linear vector fields with the singularity of saddle type and $div(X)=div(Y)=0$.
\end{itemize}
	
	
Our main results are the following.\\
	
\noindent {\bf Theorem A:} Piecewise smooth vector fields in $\mathfrak{X}^{S^1}_0\cup \mathfrak{X}^{S^1}_1$ admit at most one crossing limit cycle that intersects $S^1$ in two points.\\

\noindent {\bf Theorem B:} Piecewise smooth vector fields in $\mathfrak{X}^{S^1}_2$ admit at most two crossing limit cycles that intersect $S^1$ in two points.\\

\noindent {\bf Theorem C:} The maximum number of crossing limit cycles that intersect $S^1$ in two points for piecewise smooth vector fields in $\mathfrak{X}^{S^1}_3\cup \mathfrak{X}^{S^1}_4$ is less than or equal to two.\\

This paper is divided as follows. Section 2 presents the basic definitions about piecewise smooth vector fields according to Filippov's convention \cite{F}. In Section 3 we present the general construction employed to deduce the closing equations, whose solutions are the closing trajectories of $Z_{XY}$. In Section 4 we study the existence of periodic orbits if $X$ is a constant vector field and prove Theorem A. Section 5 is dedicated to the study of saddles and centers and to the proof of Theorem B and Theorem C.In Section 6 we present a result about he maximum number of crossing limit cycles when the submanifold of discontinuity is a ellipse. 
	
\section{Filippov convention for piecewise smooth differential equations}
\label{fili}
	
Let $X$ and $Y$ be smooth vector fields defined in an open and convex subset $U\subset \mathbb{R}^2$ and, without loss of generality, assume that the origin belongs to $U$. Consider $f:U\rightarrow\mathbb{R}$ a function $\mathcal{C}^r$, with $r>1$, ($\mathcal{C}^r$ denotes the set of continuously differentiable functions of order $r$), for which $0$ is a regular value. Thus, the curve  $\Sigma=f^{-1}(0)\cap U$ is a submanifold of dimension 1 and divides the open set $U$ into two open sets,
$$\Sigma^{+}=\{(x,y)\in U:f(x,y)>0\} \quad \mbox{and} \quad \Sigma^{-}=\{(x,y)\in U:f(x,y)<0\}.$$
	
A Filippov planar system is a piecewise smooth vector field defined in the following form: 
\begin{equation}\label{sist.fil}
	Z_{XY}(x,y)=\begin{cases}
		X(x,y),& (x,y) \in \Sigma^{+}, \\
		Y(x,y), & (x,y) \in \Sigma^{-},
	\end{cases}
\end{equation}
in order to identify the components of the field. Moreover, we assume that $X$ and $Y$ are fields of class $\mathcal{C}^k$ with $k>1$ in $\overline{\Sigma^{+}}$ and $\overline{\Sigma^{-}}$, respectively, where $\overline{\Sigma^{\pm}}$ denotes the closure of $\Sigma^{\pm}$.    
	
By using $\mathcal{Z}^k$ we denote the space of vector fields of this type, which can be taken as $\mathcal{Z}^k=\mathcal{X}^k\times \mathcal{X}^k$, where, by abuse of notation, $\mathcal{X}^k$ denotes the set of vector fields of class $\mathcal{C}^k$ defined in $\overline{\Sigma^{+}}$ and $\overline{\Sigma^{-}}$. We consider $\mathcal{Z}^k$ with the product topology $\mathcal{C}^k$.
	
To establish the dynamic given by a Filippov vector field $Z_{XY}$ in $U$, we need to define the local trajectory through a point $p\in U$, that is, we must define the flow $\varphi_z(t,p)$ of (\ref{sist.fil}). If $p \in \Sigma_{}^{\pm}$, then the trajectory through $p$ is given by the fields $X$ and $Y$ in the usual way. However, if $p \in \Sigma$, we must be more careful defining the trajectory. In order to extend the definition of trajectory for $\Sigma$, we are going to divide the discontinuity submanifold $\Sigma$ in the closure of three disjoint regions: 
\begin{itemize}
	\item[1.] Crossing region: $ \Sigma^{c}=\{p \in \Sigma\colon Xf(p)\cdot Yf(p) >0\}$,
	\item[2.] Sliding region: $ \Sigma^{s}=\{p \in \Sigma\colon Xf(p) <0,  Yf(p) >0\}$,
	\item[3.] Escape region: $ \Sigma^{e}=\{p \in \Sigma\colon Xf(p) >0,  Yf(p) <0\}$,
\end{itemize}
where $Xf(p)=\langle X(p), \nabla f(p)\rangle$ and $Yf(p)=\langle Y(p), \nabla f(p)\rangle$, are Lie's derivative of $f$ with respect to the field $X$ in $p$ and $f$ with respect to the field $Y$ in $p$, respectively. These three regions are open subsets of $\Sigma$ in the induced topology and can have more than one convex component. 
	
We can observe that when defining the regions above we aren't including the {tangent points}, that is, the points $p \in \Sigma$ for which $Xf(p)=0$ or $Yf(p)=0$. These points are in the boundaries of the regions $\Sigma^{c}$, $\Sigma^{s}$ and $\Sigma^{e}$, which are going to be denoted by $\partial \Sigma^{c}$, $\partial \Sigma^{s}$ and $\partial\Sigma{}^{e}$, respectively.
	
Note that if $X(p)=0$, then $Xf(p)=0$, so the critical points of $X$ in $\Sigma$ are also included in the tangent points. Now, if $X(p)\neq 0$ and $Xf(p)=0$, we confirm that the trajectory of $X$ passing through $p$ is, indeed, tangent to $\Sigma$.
	
We can distinguish the tangency types between a smooth field and a manifold depending on how the contact between the trajectory of the field and the manifold occurs. Next, we define two types of tangency.
	
\begin{definition}
	A smooth vector field $X$ admits a fold or quadratic tangency with $\Sigma = \{(x,y)\in U: f(x,y)=0\}$ in a point $p\in \Sigma$ if $Xf(p)=0$ and $X^2f(p)\neq 0$, being $X^2f(p)=\langle X(p), \nabla Xf(p)\rangle$.
\end{definition}
	
\begin{definition}
	A smooth vector field $X$ admits a cusp or cubic tangency with $\Sigma = \{(x,y)\in U: f(x,y)=0\}$ in a point $p\in \Sigma$ if $Xf(p)=X^2f(p)=0$ and $X^3f(p)\neq 0$, being $X^3f(p)=\langle X(p), \nabla X^2f(p)\rangle$.
\end{definition}
	
Let's define the trajectory passing through a point $p$ in $\Sigma^{c}$, $\Sigma^{s}$ and $\Sigma^{e}$.  
	
If $p \in \Sigma^{c}$, both the vector fields $X$ and $Y$ point to $\Sigma^{+}$ or $\Sigma^{-}$ and, therefore, it is sufficient to concatenate the trajectories of $X$ and $Y$ that pass through $p$. 
	
If $p\in \Sigma^{s}\cup \Sigma^{e}$, we have that the vector fields point to opposite directions, thus, we can't concatenate the trajectories. In this way, the local orbit is provided by the Filippov convention. Thus, we define the { sliding vector field}
\begin{equation}\label{campo_deslizante}
	Z^{s}(p)=\frac{1}{Yf(p)-Xf(p)}(Yf(p)X(p)-Xf(p)Y(p)).
\end{equation}  
	
Note that $Z^s$ represents the convex linear combination of $X(p)$ and $Y(p)$ so that $Z^{s}$ is tangent to $\Sigma$, moreover, its trajectories are contained in $\Sigma^{s}$ or $\Sigma^{e}$. Thus, the trajectory through $p$ is the trajectory defined by the sliding vector field in (\ref{campo_deslizante}).
	
\section{Deduction of the closing equations in the general case}
	
The following result allows us to study the existence of crossing limit cycles that intersect $S^1$ in two points for all the sets described previously. 
	
\begin{proposition}\label{clossingequation}
	Let $X,Y\in \mathfrak{X}^r(U)$ be nonzeros, with $X(x,y)=(\eta+ax+by,\zeta+cx+dy)$, $Y(x,y)=(\delta+lx+ky,\varepsilon +mx+ny)$ and $div X= div Y =0$. Then, the system of closing equations to study the existence of limit cycles that intersect the discontinuity manifold $S$ in two distinct points $(p,q)$ and $(r,s)$ is provided by
	$$\left\lbrace \def\arraystretch{1.5} \begin{array}{l}
		-\zeta p +\eta q +apq-\frac{cp^2}{2}+\frac{bq^2}{2}=-\zeta r +\eta s +ars-\frac{cr^2}{2}+\frac{bs^2}{2},\\
		-\varepsilon p +\delta q +lpq-\frac{mp^2}{2}+\frac{kq^2}{2}=-\varepsilon r +\delta s +lrs-\frac{mr^2}{2}+\frac{ks^2}{2},\\
		p^2+q^2=1,\\
		r^2+s^2=1.
	\end{array}\right. $$
\end{proposition} 
	
\begin{proof}
	We know that $X(x,y)=(\eta+ax+by,\zeta+cx+dy)$ with $divX=0$, hence 
	$$divX=\frac{\partial (\eta +ax+by)}{\partial x}+\frac{\partial (\zeta +cx+dy)}{\partial y}=a+d=0 \Rightarrow d=-a,$$
	thus, $X(x,y)=(\eta+ax+by,\zeta+cx-ay)$.
		
	Consider $H_1(x,y)=h_{00}+h_{10}x+h_{11}xy+h_{01}y+h_{10}x^2+h_{01}y^2$, we have 
	$$\nabla H_1(x,y)=(h_{10}+h_{11}y+2h_{20}x,h_{01}+h_{11}x+2h_{02}y), $$
	thus,
	$$\begin{array}{lll}
	\left\langle \nabla H_1, X \right\rangle & = & \left\langle (h_{10}+h_{11}y+2h_{20}x,h_{01}+h_{11}x+2h_{02}y), (\eta+ax+by,\zeta+cx-ay) \right\rangle \\[0.5cm] 
	&=&(h_{10}\eta+h_{01}\zeta)+x(h_{10}a+2h_{02}\eta+h_{01}c+h_{11}\zeta)+y(h_{10}b+h_{11}\eta-h_{01}a+2h_{02}\zeta)\\[0.5cm]
	& & +xy(h_{11}a+2h_{20}b-h_{11}a+2h_{02}c)+x^2(2h_{20}a+h_{11}c)+y^2(h_{11}b-2h_{02}a).
	\end{array}$$	
	We want that $\left\langle \nabla H_1, X \right\rangle=0$, then we obtain the system
	$$\left\lbrace \begin{array}{l}
	h_{10}\eta+h_{01}\zeta=0,\\
	h_{10}a+2h_{02}\eta+h_{01}c+h_{11}\zeta=0,\\
	h_{10}b+h_{11}\eta-h_{01}a+2h_{02}\zeta=0,\\
	2h_{20}b+2h_{02}c=0, \\
	2h_{20}a+h_{11}c)+y^2(h_{11}b-2h_{02}a=0.
	\end{array}	\right. $$
	Solving this system we have 
	$$h_{10}=\frac{2h_{20}\zeta}{c}, \quad h_{01}=-\frac{2h_{20}\eta}{c},\quad h_{02}=-\frac{h_{20}b}{c}\quad \mbox{and} \quad h_{11}=-\frac{2h_{20}a}{c}.$$
	Taking $h_{20}=-\frac{c}{2}$, we obtain that the first integral of the field $X$ is 
	$$H_1(x,y)=-\zeta x+\eta y+axy -\frac{cx^2}{2}+\frac{by^2}{2}.$$
	In a similar way, we determine that the first integral of the field $Y$ is
	$$H_2(x,y)=-\varepsilon x+\delta y+lxy -\frac{mx^2}{2}+\frac{ky^2}{2}.$$
		
	Therefore, the system that we obtain from the closing equations to study the existence of crossing limit cycles that intersect the discontinuity manifold $S$ in two distinct points $(p,q)$ and $(r,s)$ is provided by:
	$$\left\lbrace \def\arraystretch{1.5} \begin{array}{l}
		H_1(p,q)=H_1(r,s),\\
		H_2(p,q)=H_2(r,s),\\
		h(p,q)=0,\\
		h(r,s)=0.
	\end{array}\right.\Rightarrow
	\left\lbrace \def\arraystretch{1.5} \begin{array}{l}
		-\zeta p +\eta q +apq-\frac{cp^2}{2}+\frac{bq^2}{2}=-\zeta r +\eta s +ars-\frac{cr^2}{2}+\frac{bs^2}{2},\\
		-\varepsilon p +\delta q +lpq-\frac{mp^2}{2}+\frac{kq^2}{2}=-\varepsilon r +\delta s +lrs-\frac{mr^2}{2}+\frac{ks^2}{2},\\
		p^2+q^2=1,\\
		r^2+s^2=1.
	\end{array}\right.$$
\end{proof}

\section{Existence of periodic orbits in the presence of constant differential equations}

In this section we study the existence of crossing limit cycles when one of the vector fields is constant. First we demonstrate Theorem A, then we are going to present examples of each case in the next subsections.

\begin{proof}[Proof of Theorem A]
Consider $X(x,y)=(\eta,\zeta)$ and $Y(x,y)=(\delta+lx+ky,\varepsilon +mx+ny)$ with $divY=0$, we have that $Z_{XY}\in \mathfrak{X}^{S^1}_0\cup \mathfrak{X}^{S^1}_1$, by the Proposition \ref{clossingequation} the system of closing equations to study the existence of limit cycles that intersect $S^1$ in two points $(p,q)$ and $(r,s)$ is given by 
$$\left\lbrace \def\arraystretch{1.5} \begin{array}{l}
	-\zeta p +\eta q =-\zeta r +\eta s,\\
	-\varepsilon p +\delta q +lpq-\frac{mp^2}{2}+\frac{kq^2}{2}=-\varepsilon r +\delta s +lrs-\frac{mr^2}{2}+\frac{ks^2}{2},\\
	p^2+q^2=1,\\
	r^2+s^2=1.
\end{array}\right. $$
To solve this system, we use the Grobener base, with which we obtain two solution
$(p,q,r,s)$ and $(r,s,p,q)$,
being
$$ \begin{array}{lll}
	p&=&\dfrac{1}{\zeta  \left(\zeta ^2+\eta ^2\right) \left(\zeta  \eta  (k+m)+l \left(\eta ^2-\zeta ^2\right)\right)^2}\left(\delta  \zeta ^6 \eta  k+\delta  \zeta ^4 \eta ^3 k-\zeta ^5 \eta ^2 k \epsilon -\zeta ^3 \eta ^4 k \epsilon +\delta  \zeta ^7 (-l)\right. \\[0.5cm]
	& &\left. +\delta  \zeta ^3 \eta ^4 l+\zeta ^6 \eta  l \epsilon -\zeta ^2 \eta ^5 l \epsilon +\delta  \zeta ^6 \eta  m+\delta  \zeta ^4 \eta ^3 m-\zeta ^5 \eta ^2 m \epsilon -\zeta ^3 \eta ^4 m \epsilon +\eta  \left[ \zeta ^2 \left(\zeta ^2+\eta ^2\right) \right.\right.  \\[0.5cm]
	& &  \left(\zeta  \eta  (k+m)+l \left(\eta ^2-\zeta ^2\right)\right)^2 \left(-\delta ^2 \zeta ^2 \left(\zeta ^2+\eta ^2\right)+2 \delta  \zeta  \eta  \epsilon  \left(\zeta ^2+\eta ^2\right)+\eta ^2 \left(\zeta ^2 k^2+2 \zeta ^2 k m \right.\right.  \\[0.5cm]
	& &\left.  \left. \left.\left.  +\zeta ^2 m^2+\epsilon ^2 \left(-\left(\zeta ^2+\eta ^2\right)\right)\right)-2 \zeta  \eta  l \left(\zeta ^2-\eta ^2\right) (k+m)+l^2 \left(\zeta ^2-\eta ^2\right)^2\right)\right]^{\frac{1}{2}}  \right) 
\end{array}$$
	
$$\begin{array}{lll}
	q&=&\dfrac{1}{\left(\zeta ^2+\eta ^2\right) \left(\zeta  \eta  (k+m)+l \left(\eta ^2-\zeta ^2\right)\right)^2}\left(-\delta  \zeta ^4 \eta ^2 k-\delta  \zeta ^2 \eta ^4 k+\zeta ^3 \eta ^3 k \epsilon +\zeta  \eta ^5 k \epsilon +\delta  \zeta ^5 \eta  l-\delta  \zeta  \eta ^5 l\right. \\[0.5cm]
	& &\left. \left. -\zeta ^4 \eta ^2 l \epsilon +\eta ^6 l \epsilon -\delta  \zeta ^4 \eta ^2 m-\delta  \zeta ^2 \eta ^4 m+\zeta ^3 \eta ^3 m \epsilon +\zeta  \eta ^5 m \epsilon +\left[ \zeta ^2 \left(\zeta ^2+\eta ^2\right) \left(\zeta  \eta  (k+m)\right. \right. \right. \right. \\[0.5cm]
	& &\left. \left. +l \left(\eta ^2-\zeta ^2\right)\right)^2 \left(-\delta ^2 \zeta ^2 \left(\zeta ^2+\eta ^2\right)+2 \delta  \zeta  \eta  \epsilon  \left(\zeta ^2+\eta ^2\right)+\eta ^2 \left(\zeta ^2 k^2+2 \zeta ^2 k m+\zeta ^2 m^2\right. \right.\right.  \\[0.5cm]
	& &\left. \left. \left. \left. +\epsilon ^2 \left(-\left(\zeta ^2+\eta ^2\right)\right)\right)-2 \zeta  \eta  l \left(\zeta ^2-\eta ^2\right) (k+m)+l^2 \left(\zeta ^2-\eta ^2\right)^2\right)\right]^{\frac{1}{2}} \right) 
	\end{array}$$
	
$$\begin{array}{lll}
	r&=&\dfrac{1}{\zeta  \left(\zeta ^2+\eta ^2\right) \left(\zeta  \eta  (k+m)+l \left(\eta ^2-\zeta ^2\right)\right)^2}\left(\delta  \zeta ^6 \eta  k+\delta  \zeta ^4 \eta ^3 k-\zeta ^5 \eta ^2 k \epsilon -\zeta ^3 \eta ^4 k \epsilon -\delta  \zeta ^7 l+\delta  \zeta ^3 \eta ^4 l\right. \\[0.5cm]
	& &\left. \left.  +\zeta ^6 \eta  l \epsilon -\zeta ^2 \eta ^5 l \epsilon +\delta  \zeta ^6 \eta  m+\delta  \zeta ^4 \eta ^3 m-\zeta ^5 \eta ^2 m \epsilon -\zeta ^3 \eta ^4 m \epsilon -\eta \left[ \zeta ^2 \left(\zeta ^2+\eta ^2\right) \left(\zeta  \eta  (k+m)\right. \right. \right. \right. \\[0.5cm]
	& & \left. \left.  +l \left(\eta ^2-\zeta ^2\right)\right)^2 \left(-\delta ^2 \zeta ^2 \left(\zeta ^2+\eta ^2\right)+2 \delta  \zeta  \eta  \epsilon  \left(\zeta ^2+\eta ^2\right)+\eta ^2 \left(\zeta ^2 k^2+2 \zeta ^2 k m+\zeta ^2 m^2\right. \right. \right. \\[0.5cm]
	& &\left. \left. \left. \left. +\epsilon ^2 \left(-\left(\zeta ^2+\eta ^2\right)\right)\right)-2 \zeta  \eta  l \left(\zeta ^2-\eta ^2\right) (k+m)+l^2 \left(\zeta ^2-\eta ^2\right)^2\right)\right]^{\frac{1}{2}}   \right) 
\end{array}
$$
	
$$
\begin{array}{lll}
	s&=&-\dfrac{1}{\left(\zeta ^2+\eta ^2\right) \left(\zeta  \eta  (k+m)+l \left(\eta ^2-\zeta ^2\right)\right)^2}\left(\delta  \zeta ^4 \eta ^2 k+\delta  \zeta ^2 \eta ^4 k-\zeta ^3 \eta ^3 k \epsilon -\zeta  \eta ^5 k \epsilon -\delta  \zeta ^5 \eta  l+\delta  \zeta  \eta ^5 l\right. \\[0.5cm]
	& &\left. \left. +\zeta ^4 \eta ^2 l \epsilon -\eta ^6l \epsilon +\delta  \zeta ^4 \eta ^2 m+\delta  \zeta ^2 \eta ^4 m-\zeta ^3 \eta ^3 m \epsilon -\zeta  \eta ^5 m \epsilon +\left[ \zeta ^2 \left(\zeta ^2+\eta ^2\right) \left(\zeta  \eta  (k+m)\right.\right. \right. \right.   \\[0.5cm]
	& &\left. \left. +l \left(\eta ^2-\zeta ^2\right)\right)^2 \left(-\delta ^2 \zeta ^2 \left(\zeta ^2+\eta ^2\right)+2 \delta  \zeta  \eta  \epsilon  \left(\zeta ^2+\eta ^2\right)+\eta ^2 \left(\zeta ^2 k^2+2 \zeta ^2 k m+\zeta ^2 m^2\right. \right.\right.  \\[0.5cm]
	& & \left. \left. \left.\left.  +\epsilon ^2 \left(-\left(\zeta ^2+\eta ^2\right)\right)\right)-2 \zeta  \eta  l \left(\zeta ^2-\eta ^2\right) (k+m)+l^2 \left(\zeta ^2-\eta ^2\right)^2\right)\right]^{\frac{1}{2}}  \right)  
\end{array}
$$
with $\zeta \neq 0$, $\eta \neq \pm \zeta$ and $l\neq \frac{(k+m)\zeta\eta}{\zeta^2-\eta^2}$.

We note that the solutions define the same closed curves that pass through $(p,q)$ and $(r,s)$. Therefore $Z_{XY}$ admits at most one crossing limit cycle that intersects $S^1$ in two points.  
\end{proof}

\subsection{Constant-center case}

In this subsection, we show that there are vector fields in $\mathfrak{X}^{S^1}_0$ that admit an infinite number of crossing periodic orbits, vector fields that do not have periodic orbits, and vector fields that admit only one crossing limit cycle. To do this, we are going to present examples of each case.

\begin{corollary}
	There is $Z_{XY} \in \mathfrak{X}^{S^1}_0$ that admits an infinite number of crossing periodic orbits that intersect $S^1$ in two points.
\end{corollary}

\begin{proof}
	Consider the piecewise smooth vector field $Z_{XY}$, where $$X(x,y)=(1, 0)\quad \mbox{and} \quad Y(x,y)=(2 y, -6 x).$$
	The fields $X$ and $Y$ have the following first integrals  
	$$H_1(x,y)=y\quad \mbox{and}\quad H_2(x,y)=3x^2 + y^2,$$
	respectively.
	
	To study the existence of crossing limit cycles of the field $Z_{XY}$ we use the closing equations, with which we obtain the system of nonlinear equations
	$$\left\lbrace \def\arraystretch{1.5} \begin{array}{l}
		q=s,\\
		3p^2 + q^2=3r^2 + s^2,\\
		p^2+q^2=1,\\
		r^2+s^2=1,
	\end{array}\right. $$
	where $(p,q,r,s)$ satisfies $(p,q)\neq (r,s)$. From the closing equations we have $q=s$. 
	Thus, we get that $p=\pm r$. So, $(p,q,r,s)=(p,q,\pm p, q)$, and since $(p,q)\neq (r,s)$ we cannot have $p=r$. Hence, $(p,q,r,s)=(p,q,-p,q)$.
	
	Therefore, $Z_{XY}$ has an infinite number of crossing periodic orbits that intersect $S^1$ in two points, which we illustrate in Figure \ref{ccorbper}.
	
\end{proof}

\begin{corollary}
	There is $Z_{XY} \in \mathfrak{X}^{S^1}_0$ that does not have crossing periodic orbits that intersect $S^1$ in two points.
\end{corollary}

\begin{proof}
	Consider the piecewise smooth vector field $Z_{XY}$, on which $$X(x,y)=(0,-1)\quad \mbox{and} \quad Y(x,y)=(1 + 10 y, -1 - 2 x).$$
	The fields $X$ and $Y$ have the following first integrals 
	$$H_1(x,y)=x\quad \mbox{and}\quad H_2(x,y)=x + x^2 + y + 5y^2,$$
	respectively.
	
	To study the existence of crossing limit cycles of the field $Z_{XY}$ we use the closing equations, with which we obtain the system of nonlinear equations
	$$\left\lbrace \def\arraystretch{1.5} \begin{array}{l}
		p=r,\\
		p + p^2 + q + 5q^2=r + r^2 + s + 5s^2,\\
		p^2+q^2=1,\\
		r^2+s^2=1,
	\end{array}\right. $$
	where $(p,q,r,s)$ satisfies $(p,q)\neq (r,s)$. From the closing equations we have $p=r$, $q^2=1-p^2$ and $s^2=1-r^2$. Replacing this information on the second equation of the system we get
	$q=s$.
	Thus, $(p,q,r,s)=(p,q,p,q)$. 
	
	Therefore, since $(p,q)\neq (r,s)$, $Z_{XY}$ does not have crossing periodic orbits that intersect $S^1$ in two points, as we illustrate in Figure \ref{ccsemorb}.
	
\end{proof}

\begin{corollary}
\label{st332}
	There is $Z_{XY} \in \mathfrak{X}^{S^1}_0$ that has exactly one crossing limit cycle that intersects $S^1$ in two points.
\end{corollary}

\begin{proof}
	Consider the piecewise smooth vector field $Z_{XY}$, where
	$$X(x,y)=(2,-1)\quad \mbox{and}\quad Y(x,y)=\left(2 -x+2 y,-1-4 x+y\right). $$
	The fields $X$ and $Y$ admit the following first integrals
	$$H_1(x,y)=x+2y\quad \mbox{and}\quad H_2(x,y)=x+2y+2 x^2-x y+y^2,$$
	respectively. 
	
	To study the existence of crossing limit cycles of the field $Z_{XY}$ we use the closing equations, with which we obtain the system of nonlinear equations
	$$\left\lbrace \def\arraystretch{1.5} \begin{array}{l}
		p+2q+2 p^2-p q+q^2=r+2s+2 r^2-r s+s^2,\\
		p+2q=r+2s,\\
		p^2+q^2=1,\\
		r^2+s^2=1,
	\end{array}\right. $$
	on which $(p,q,r,s)$ satisfies $(p,q)\neq (r,s)$. To solve this system, we used the Grobner base, obtaining the solutions
	$$\left(-\frac{2}{\sqrt{5}},\frac{1}{\sqrt{5}},\frac{2}{\sqrt{5}},-\frac{1}{\sqrt{5}} \right) \quad \mbox{and} \quad \left(\frac{2}{\sqrt{5}},-\frac{1}{\sqrt{5}} ,-\frac{2}{\sqrt{5}},\frac{1}{\sqrt{5}}\right), 
	$$
	that represent the same closed curve. 
	
	To complete, we check if these points are in the crossing region. Remember that $h(x,y)=x^2+y^2-1$, so $\nabla h=(2x,2y)$. Thus,
	$$Xh(x,y)=\langle \left (-2,1), (2x,2y) \right\rangle = 4x-2y,$$
	$$Yh(x,y)=\left\langle  \left( 2 -x+2 y,-1-4 x+y\right), (2x,2y) \right\rangle =4 x-2 y -2 x^2-4 x y+2 y^2.$$  
	Hence, 
	$$Xh(p,q)= -2 \sqrt{5}<0\quad \mbox{and}\quad Yh(p,q) = \frac{2}{5}-2 \sqrt{5}<0\;\; \Rightarrow\;\; Xh(p,q)Yh(p,q)>0,$$
	$$Xh(r,s) = 2 \sqrt{5}>0 \quad \mbox{and}\quad Yh(r,s) = \frac{2}{5}+2 \sqrt{5}>0\;\; \Rightarrow\;\; Xh(r,s)Yh(r,s)>0,$$
	which implies that $(p,q)$ and $(r,s)$ belong to the crossing region of the discontinuity manifold $S^1$.
	
	Therefore, $Z_{XY}$ has only one crossing limit cycle that intersects $S^1$ in two points, which we illustrate in Figure \ref{cc1ciclo} .
	
\end{proof}

\begin{figure}[h]
	\begin{subfigure}{0.33\textwidth}
		\centering
		\includegraphics[scale=0.3]{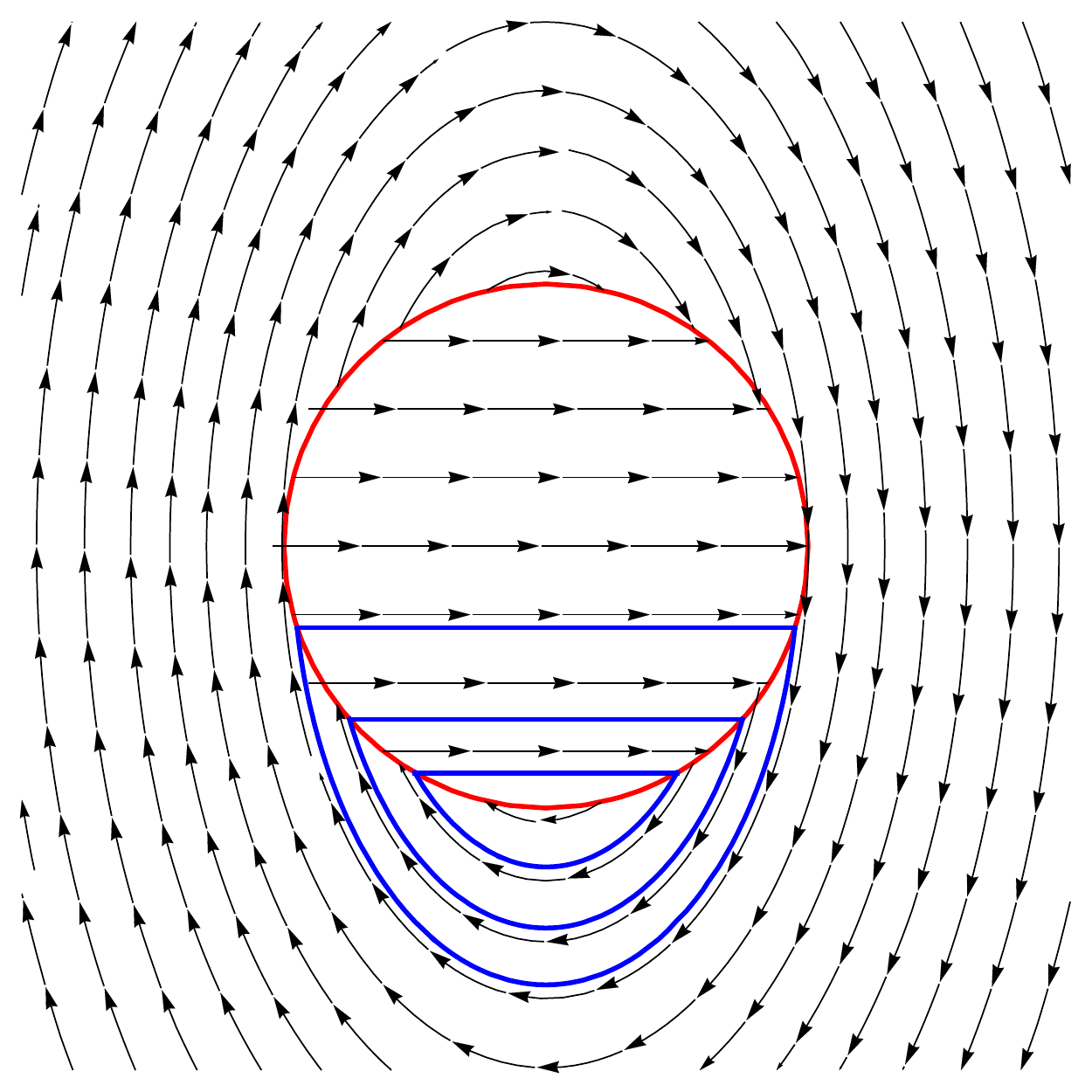}
		\caption{Periodic orbits}
		\label{ccorbper}
	\end{subfigure}%
	\begin{subfigure}{0.33\textwidth}
		\centering
		\includegraphics[scale=0.3]{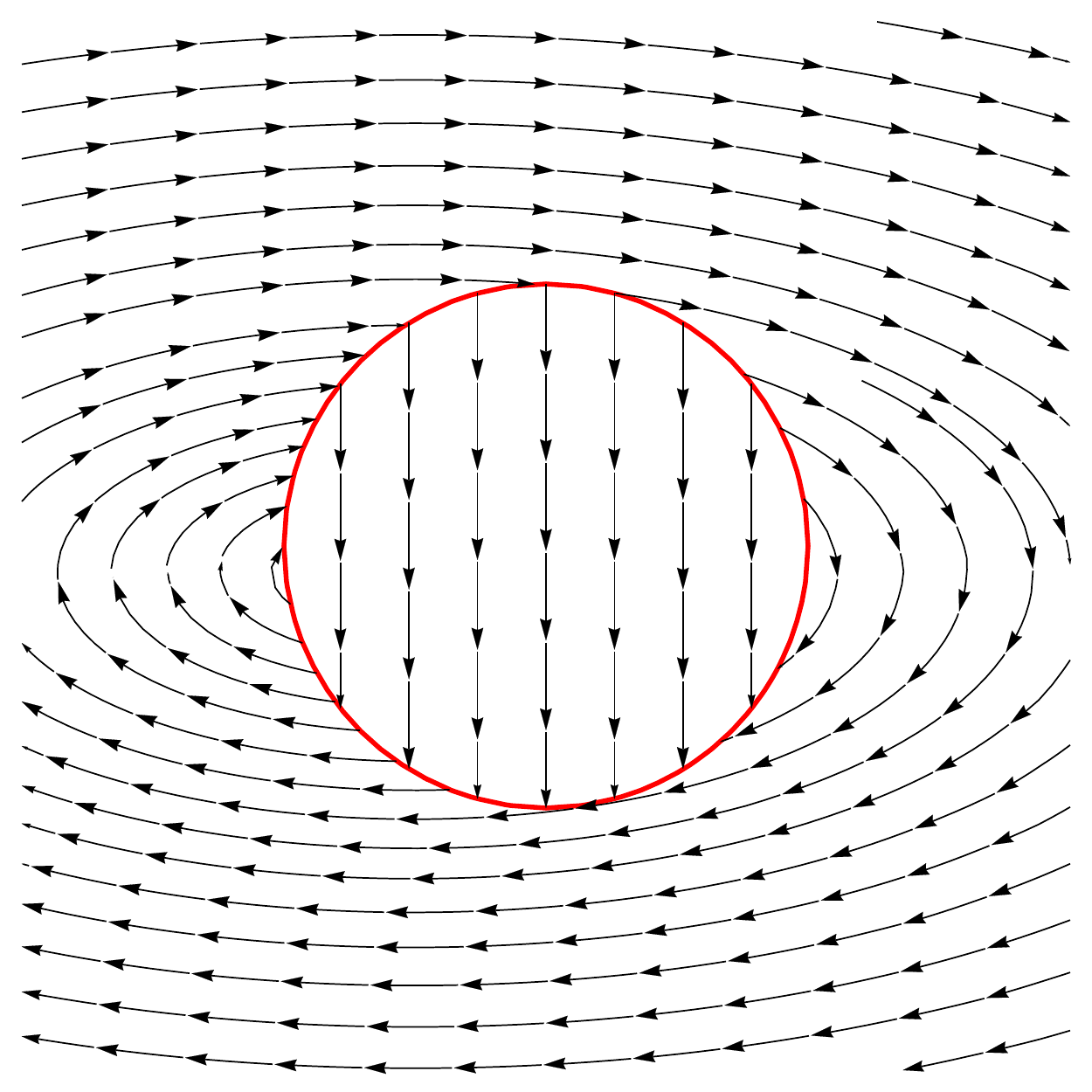}
		\caption{Without periodic orbits}
		\label{ccsemorb}
	\end{subfigure}%
	\begin{subfigure}{0.33\textwidth}
		\centering
		\includegraphics[scale=0.3]{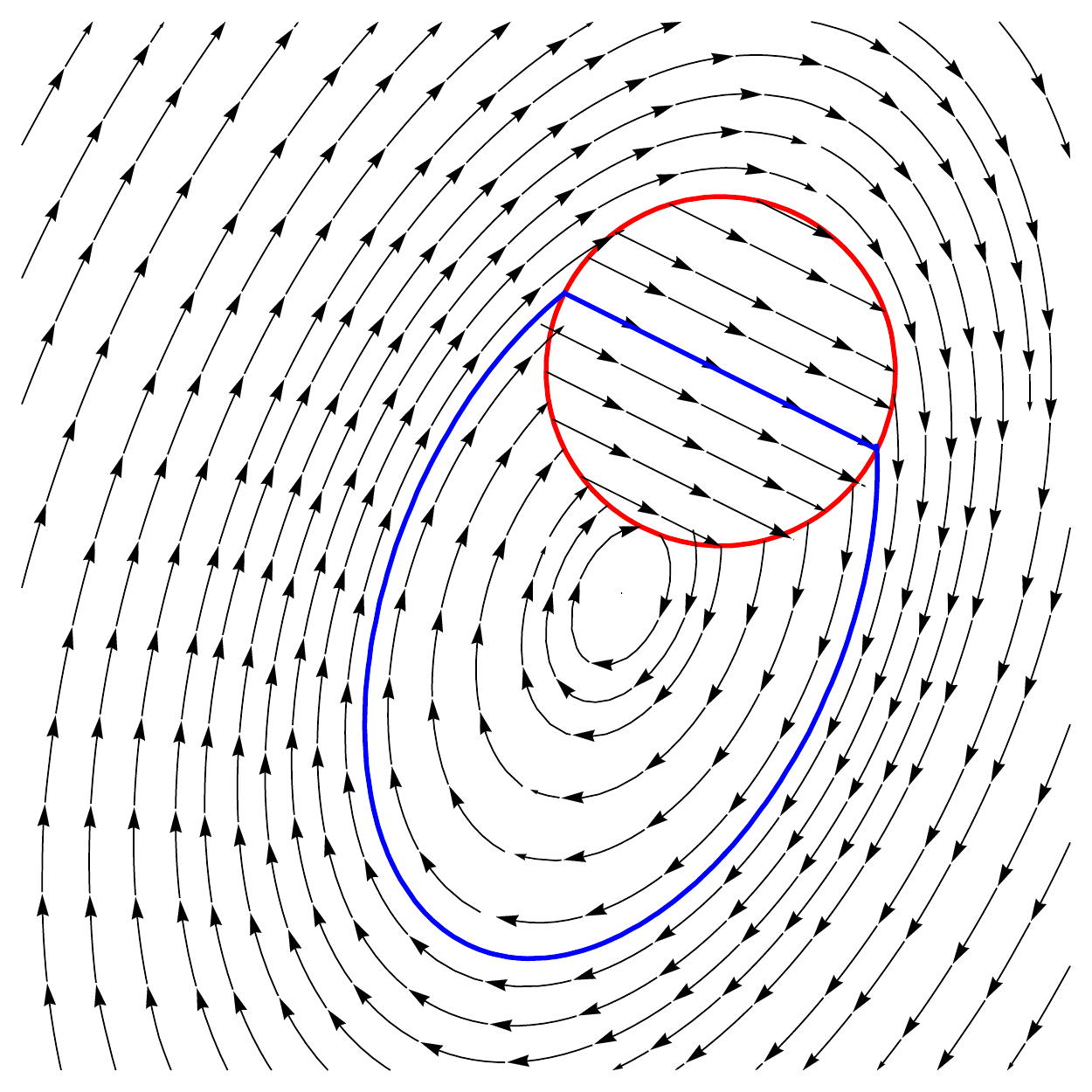}
		\caption{ One limit cycle}
		\label{cc1ciclo}
	\end{subfigure}%
	\caption{Phase portrait of vector fields $Z_{XY}\in \mathfrak{X}^{S^1}_0$.}
\end{figure}

\begin{corollary}\label{unst22} Let $\gamma$ the limit cycle constructed in the proof of Corollary \ref{st332}. Then $\gamma$ is an unstable limit cycle.
\end{corollary}
\begin{proof} Let $\mathcal E_1$ be the ellipsis $$x + 2 y + 2 x^2 - xy + y^2=11/5$$ and $\mathcal L$ be the line $$x+2y=0.$$

The limit cycle $\gamma$ is obtained by the union of the part of $\mathcal E_1$ that lies outside the unitary circle with the part of the line $\mathcal L$ that lies inside the unitary circle. These curves intersect the circle at the points $p=(-2/\sqrt{5},1/\sqrt{5})$ and $q=(2/\sqrt{5},-1/\sqrt{5})$, that are in the crossing region, producing the limit cycle (the positive $t$ direction is clockwise). Now we prove that this limit cycle is unstable by using an elementary geometric construction.

The major axis of $\mathcal E_1$ is the line $\mathcal I$ given by
$$y= \dfrac{7 x-9 \sqrt{2}+13}{7 \left(\sqrt{2}-1\right)}$$
and  the perpendicular bisector of the line $\mathcal L$ is the line $\mathcal J$ given by $$y-2x=0.$$

A point $p_\delta$ in a small neighborhood of $p$ in $S^1$ can be parametrized as
$$p_\delta=\left(\frac{\delta -\frac{2}{\sqrt{5}}}{\sqrt{\left| \delta -\frac{2}{\sqrt{5}}\right| ^2+\frac{1}{5}}},\frac{1}{\sqrt{5} \sqrt{\left| \delta -\frac{2}{\sqrt{5}}\right| ^2+\frac{1}{5}}}\right).$$

Now we construct the positive trajectory of $Z_{XY}$ (see proof of Corollary \ref{st332}) by $p_\delta$. The dynamics of $Z_{XY}$ is the dynamics of $X$, that has first integral $$H_1(x,y)=x+2y$$ and points inward $S^1$.

By solving $H_1(x,y)=H_1(p_\delta)$ for $(x,y)\in S^1$ and $(x,y)\neq p_\delta$ we obtain the point $q_\delta$, in a neighborhood of $q$ in $S^1$. The segment $\overline{p_\delta q_\delta}$ is part of the trajectory of $Z_{XY}$ through $p_\delta$ and to continue the trajectory we have to switch to the vector field $Y$, that has first integral $$H_2(x,y)=x+2y+2 x^2-x y+y^2.$$

By solving $H_2(x,y)=H_2(q_\delta)$ for $(x,y)\in S^1$ and $(x,y)\neq q_\delta$ we obtain the point $r_\delta$, in a neighborhood of $p_\delta$ in $S^1$. The trajectory from $q_\delta$ to $r_\delta$ is a arc of the ellipsis $\mathcal E_2$ given by $H_2(x,y)=H_2(q_\delta)$ and the Poincaré map near $p$ is given by $$p_\delta\mapsto r_\delta.$$ See Figure \ref{figunstable} for a sketch of the construction.

To prove that $\gamma$ is unstable it is suffice to prove that $$||r_\delta-p||>||p_\delta-p||.$$ Note that the major axis of ellipsis $\mathcal E_1$ and $\mathcal E_2$ are the same (line $\mathcal I$), and due to the symmetry of the ellipsis $\mathcal E_1$ and $\mathcal E_2$ with respect to their axis, and using the fact that angle $\theta$ between $\mathcal{L}$ and $\mathcal{J}$ is positive (measured counterclockwise), the point $r_\delta$ is farther from $p$ than is $p_\delta$.

\begin{figure}[h]
	\centering
	\begin{overpic}[scale=0.5,unit=1mm]{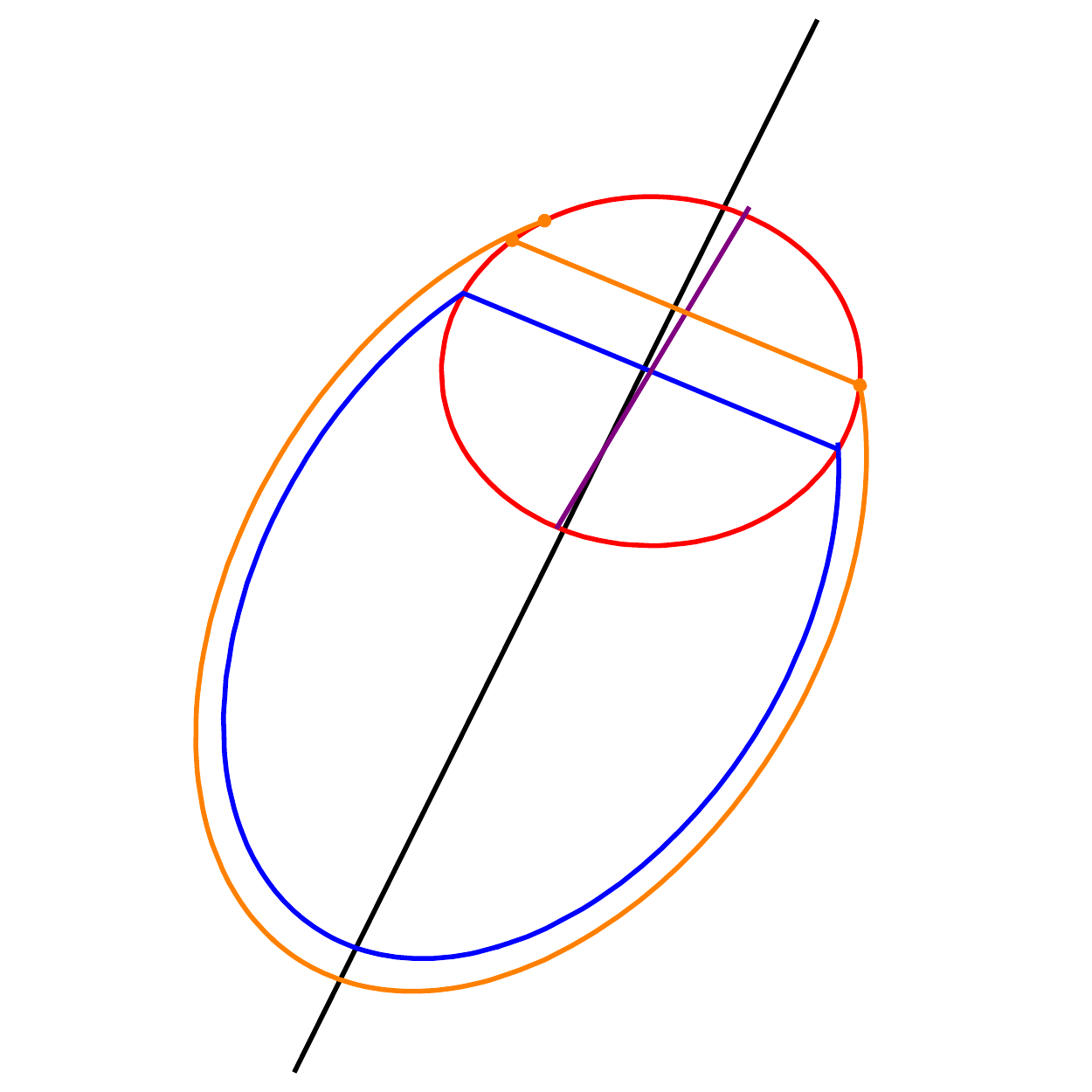}
		\put(22,30){$\mathcal E_1$}
		\put(14,20){$\mathcal E_2$}
		\put(65,57){$\mathcal L$}
		\put(48,36){$\mathcal I$}
		\put(66,74){$\mathcal J$}
		\put(46,74){\footnotesize{$p_{\delta}$}}
		\put(50,82){\footnotesize{$r_{\delta}$}}
		\put(80,65){\footnotesize{$q_{\delta}$}}
	\end{overpic}
	\caption{Unstable limit cycle in Corollaries \ref{st332} and \ref{unst22}.}
	\label{figunstable}
\end{figure}
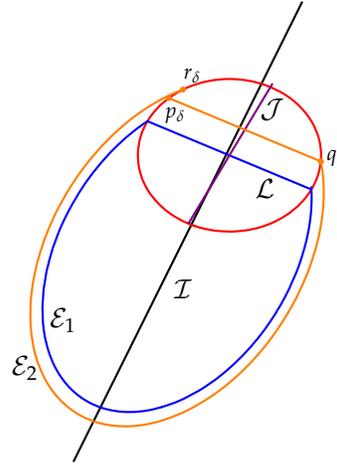

\end{proof}

\subsection{Constant-saddle case}

In this subsection, we show that there are vector fields in $\mathfrak{X}^{S^1}_1$ that admit an infinite number of crossing periodic orbits, vector fields that do not have periodic orbits, and vector fields that admit only one crossing limit cycle. To do this, we are going to present examples of each case.

\begin{corollary}
	There is $Z_{XY} \in \mathfrak{X}^{S^1}_1$ that admits an infinite number of crossing periodic orbits that intersect $S^1$ in two points.
\end{corollary}

\begin{proof}
	
	Consider the piecewise smooth vector field $Z_{XY}$, on which $$X(x,y)=(1, 0)\quad \mbox{and} \quad Y(x,y)=(4 - 2 y, -6 x).$$
	The fields $X$ and $Y$ have the following first integrals  
	$$H_1(x,y)=y\quad \mbox{and}\quad H_2(x,y)=3x^2 + 4y - y^2,$$
	respectively.
	
	To study the existence of crossing limit cycles of the field $Z_{XY}$ we use the closing equations, with which we obtain the system of nonlinear equations
	$$\left\lbrace \def\arraystretch{1.5} \begin{array}{l}
		q=s,\\
		3p^2 + 4q - q^2=3r^2 + 4s - s^2,\\
		p^2+q^2=1,\\
		r^2+s^2=1,
	\end{array}\right. $$
	where $(p,q,r,s)$ satisfies $(p,q)\neq (r,s)$. From the closing equations we have $q=s$. 
	Thus, we get that $p=\pm r$. In this way, $(p,q,r,s)=(p,q,\pm p, q)$, and since $(p,q)\neq (r,s)$ we cannot have $p=r$. So, $(p,q,r,s)=(p,q,-p,q)$.
	
	Therefore, $Z_{XY}$ admits an infinite number of crossing periodic orbits that intersect $S^1$ in two points, which we illustrate in Figure \ref{csorbper}.
\end{proof}

\begin{corollary}
	There is $Z_{XY} \in \mathfrak{X}^{S^1}_1$ that does not have crossing periodic orbits that intersect $S^1$ in two points.
\end{corollary}

\begin{proof}
	
	Consider the piecewise smooth vector field $Z_{XY}$, on which $$X(x,y)=(0,-1)\quad \mbox{and} \quad Y(x,y)=(3 - 2y, 1 - 4x).$$
	The fields $X$ and $Y$ have the following first integrals
	$$H_1(x,y)=x\quad \mbox{and}\quad H_2(x,y)=-x + 2x^2 + 3y - y^2,$$
	respectively.
	
	To study the existence of crossing limit cycles of the field $Z_{XY}$ we use the closing equations, with which we get the system of nonlinear equations
	$$\left\lbrace \def\arraystretch{1.5} \begin{array}{l}
		p=r,\\
		-p + 2p^2 + 3q - q^2=-r + 2r^2 + 3s - s^2,\\
		p^2+q^2=1,\\
		r^2+s^2=1,
	\end{array}\right. $$
	where $(p,q,r,s)$ satisfies $(p,q)\neq (r,s)$. From the closing equations we have $p=r$, $q^2=1-p^2$ and $s^2=1-r^2$. Substituting this in the second equation of the system we obtain $p=s$.
	That way, $(p,q,r,s)=(p,q,p,q)$. 
	
	Therefore, since $(p,q) \neq (r,s)$, $Z_{XY}$ does not have crossing periodic orbits that intersect $S^1$ in two points, which we illustrate in Figure \ref{cssemorb}.
	
\end{proof}

\begin{corollary}
	There is $Z_{XY} \in \mathfrak{X}^{S^1}_1$ that has exactly one crossing limit cycle that intersects $S^1$ in two points.
\end{corollary}

\begin{proof}
	
	Consider the piecewise smooth vector field $Z_{XY}$, on which
	$$X(x,y)=\left( -1,-\frac{1}{3}\right) \quad \mbox{and}\quad Y(x,y)=\left(5+2 x-2 y,1-6 x-2 y\right). $$
	The fields $X$ and $Y$ have the following first integrals
	$$H_1(x,y)=\frac{x}{3}-y\quad \mbox{and}\quad H_2(x,y)=-x+5y+3x^2+2xy-y^2,$$
	respectively. 
	
	To study the existence of crossing limit cycles of the field $Z_{XY}$ we use the closing equations, with which we obtain the system
	$$\left\lbrace \def\arraystretch{1.5} \begin{array}{l}
		-p+5q+3p^2+2pq-q^2=-r+5s+3r^2+2rs-s^2,\\
		\frac{p}{3}-q=\frac{p}{3}-q,\\
		p^2+q^2=1,\\
		r^2+s^2=1,
	\end{array}\right. $$
	where $(p,q,r,s)$ satisfies $(p,q)\neq (r,s)$. To solve this system, we used the Grobner base, rendering the solutions
	$$\left(\frac{1}{20} \left(-3 \sqrt{15}-5\right),\frac{1}{4}\left(3 -\sqrt{\frac{3}{5}}\right) ,\frac{1}{20} \left(3 \sqrt{15}-5\right),\frac{1}{20} \left(\sqrt{15}+15\right) \right),$$  
	$$\left(\frac{1}{20} \left(3 \sqrt{15}-5\right),\frac{1}{20} \left(\sqrt{15}+15\right),\frac{1}{20} \left(-3 \sqrt{15}-5\right),\frac{1}{4}\left(3 -\sqrt{\frac{3}{5}}\right) \right),$$  
	that define the same closed curve. 
	
	To complete, we check if these points are in the crossing region. Remember that $h(x,y)=x^2+y^2-1$, then $\nabla h=(2x,2y)$. So,
	$$Xh(x,y)=\left\langle  \left( -1,-\frac{1}{3}\right), (2x,2y) \right\rangle = -2x-\frac{2y}{3},$$
	$$Yh(x,y)=\left\langle  \left(5+2 x-2 y,1-6 x-2 y\right), (2x,2y) \right\rangle =10 x+2 y+4 x^2-16 x y-4 y^2.$$  
	Thus, 
	$$Xh(p,q)= \sqrt{\frac{5}{3}}>0 \quad \mbox{and}\quad Yh(p,q) = \frac{3}{5} \left(\sqrt{15}-1\right)>0\;\; \Rightarrow\;\; Xh(p,q)Yh(p,q)>0,$$
	$$Xh(r,s) =-\sqrt{\frac{5}{3}}<0 \quad \mbox{and}\quad Yh(r,s) = -\frac{3}{5} \left(\sqrt{15}+1\right)<0\;\; \Rightarrow\;\; Xh(r,s)Yh(r,s)>0,$$
	that is, $(p,q)$ and $(r,s)$ belong to the crossing region of the discontinuity manifold $S^1$.
	
	Therefore, $Z_{XY}$ has only one crossing limit cycle that intersects $S^1$ in two points, which we illustrate in Figure \ref{cs1ciclo} .
	
\end{proof}

\begin{figure}[h]
	\begin{subfigure}{0.33\textwidth}
		\centering
		\includegraphics[scale=0.3]{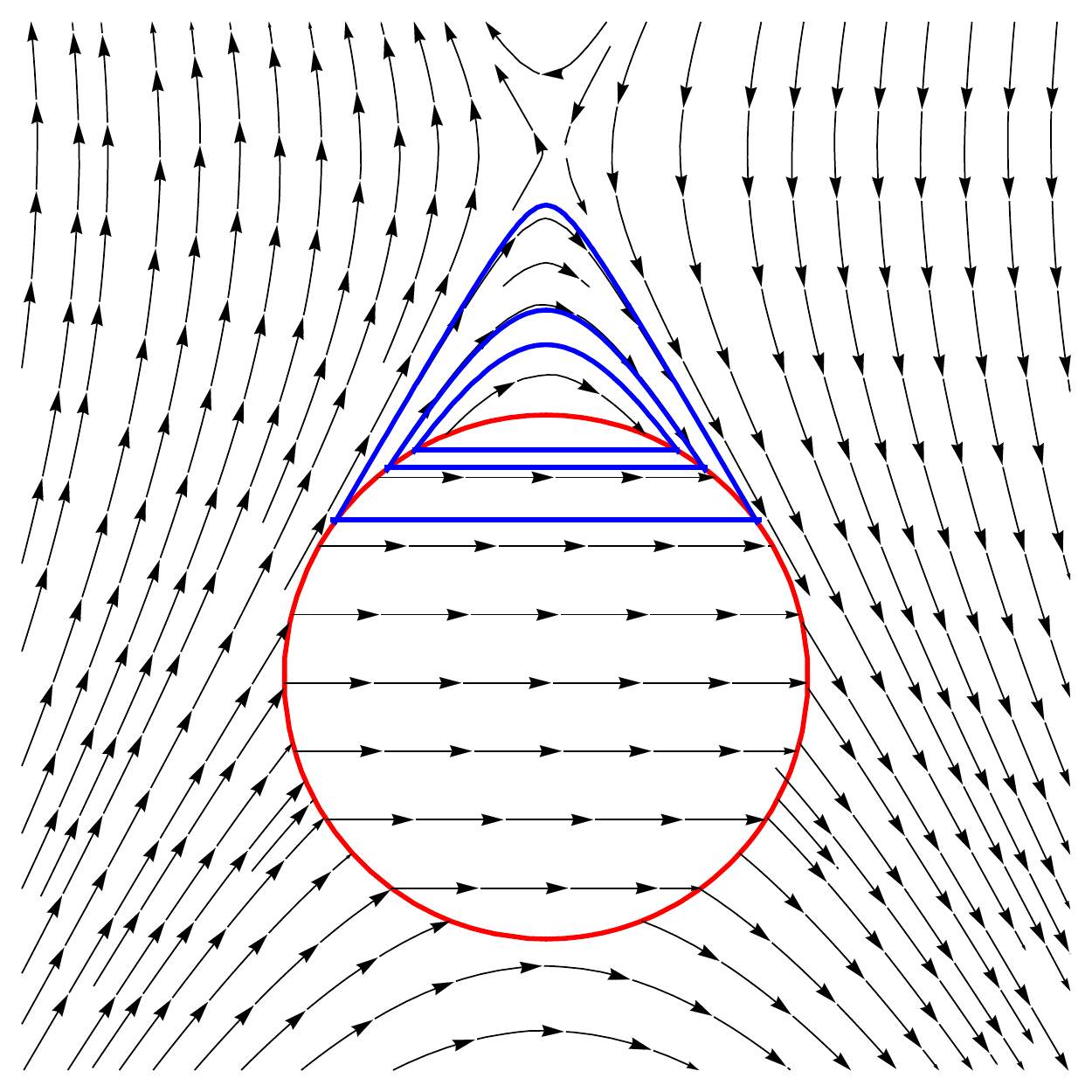}
		\caption{Periodic orbits}
		\label{csorbper}
	\end{subfigure}%
	\begin{subfigure}{0.33\textwidth}
		\centering
		\includegraphics[scale=0.3]{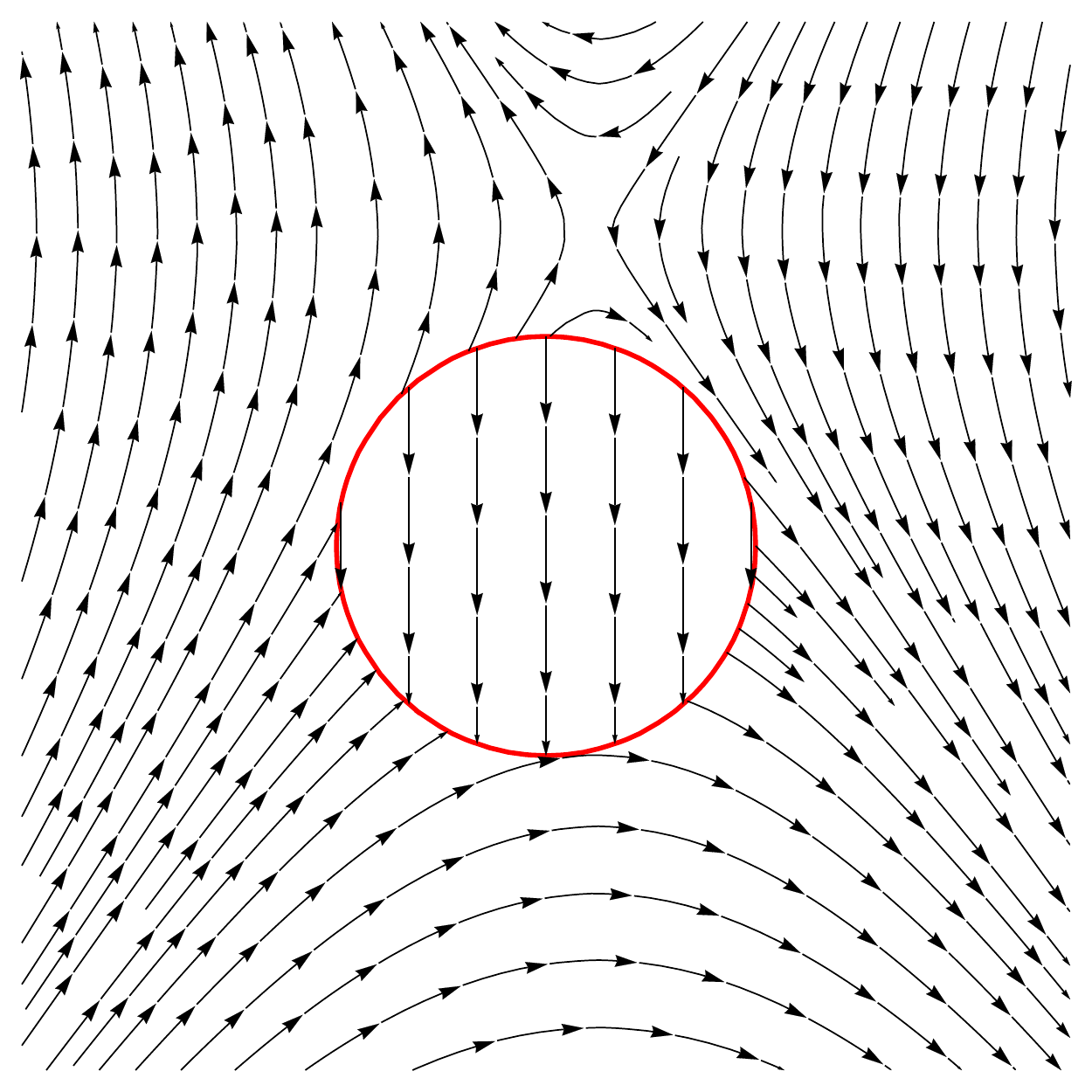}
		\caption{Without periodic orbits}
		\label{cssemorb}
	\end{subfigure}%
	\begin{subfigure}{0.33\textwidth}
		\centering
		\includegraphics[scale=0.3]{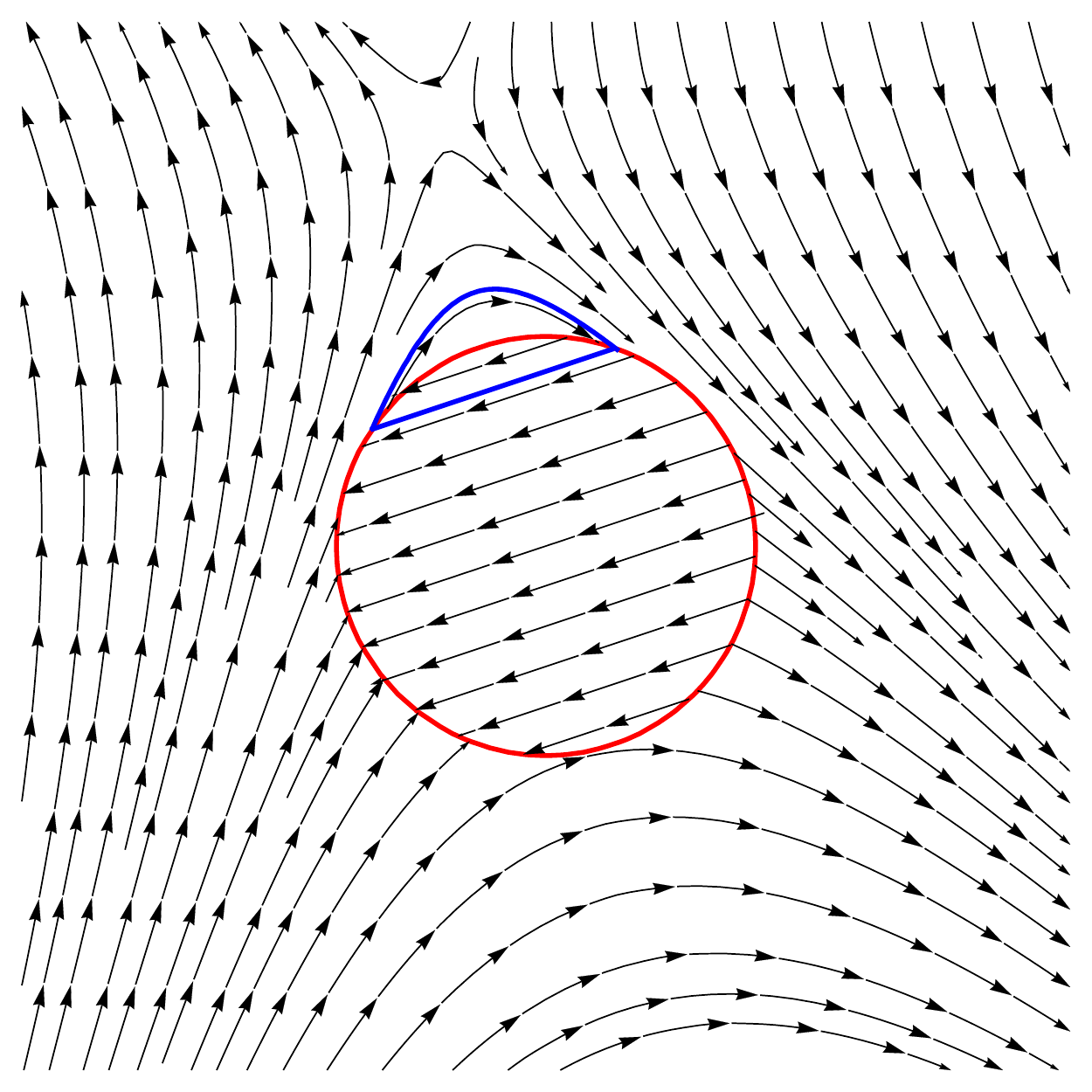}
		\caption{ One limit cycle}
		\label{cs1ciclo}
	\end{subfigure}%
	\caption{Phase portrait of vector fields $Z_{XY}\in \mathfrak{X}^{S^1}_1$.}
\end{figure} 

\section{Periodic orbits when the singularities are saddle or center}

In this section we study the existence of crossing limit cycles of piecewise smooth vector fields $Z_{XY}\in \mathfrak{X}^{S^1}_2 \cup \mathfrak{X}^{S^1}_3 \cup \mathfrak{X}^{S^1}_4$. We are going to study each case separately, but first we demonstrate the following theorem.

\begin{theorem}\label{twocycle}
	Consider $X(x,y)=(\eta+ax+by,\zeta+cx+dy)$, $Y(x,y)=(\delta+lx+ky,\varepsilon +mx+ny)$ and $divX=divY=0$. Then the maximum number of crossing limit cycles of the piecewise smooth vector field $Z_{XY}$ that intersect $S^1$ in two points is less than or equal to two.
\end{theorem}

\begin{proof}
	Consider $X(x,y)=(\eta+ax+by,\zeta+cx+dy)$ with $divX=0$ and $Y(x,y)=(\delta+lx+ky,\varepsilon +mx+ny)$ with $divY=0$, by the Proposition \ref{clossingequation} the system of closing equations to study the existence of limit cycles that intersect $S^1$ in two points $(p,q)$ and $(r,s)$ is given by 
	$$\left\lbrace \def\arraystretch{1.5} \begin{array}{l}
		H_1(p,q)=H_1(r,s),\\
		H_2(p,q)=H_2(r,s),\\
		h(p,q)=0,\\
		h(r,s)=0.
	\end{array}\right.\Rightarrow
	\left\lbrace \def\arraystretch{1.5} \begin{array}{l}
		-\zeta p +\eta q +apq-\frac{cp^2}{2}+\frac{bq^2}{2}=-\zeta r +\eta s +ars-\frac{cr^2}{2}+\frac{bs^2}{2},\\
		-\varepsilon p +\delta q +lpq-\frac{mp^2}{2}+\frac{kq^2}{2}=-\varepsilon r +\delta s +lrs-\frac{mr^2}{2}+\frac{ks^2}{2},\\
		p^2+q^2=1,\\
		r^2+s^2=1,
	\end{array}\right.$$
	where
	$$H_1(x,y)=-\zeta x +\eta y +axy-\frac{cx^2}{2}+\frac{by^2}{2}\quad \mbox{and} \quad H_2(x,y)=-\varepsilon x +\delta y +lxy-\frac{mx^2}{2}+\frac{ky^2}{2}.$$
	Suppose that the system admits three solutions, 
	$$(p,q,r,s)=(\cos(\alpha_i),\sin(\alpha_i),\cos(\theta_i),\sin(\theta_i)), \quad \mbox{for}\; i=1,2,3,$$
	with $\alpha_i, \theta_i \in [0, 2\pi)$ and $\alpha_i < \theta_i$ for $i=1,2,3$. 
	
	Substituting the solution $(\cos(\alpha_1),\sin(\alpha_1),\cos(\theta_1),\sin(\theta_1))$ in the second equation of the system we obtain 
	$$\begin{array}{lll}
		\varepsilon&=&\dfrac{1}{2 (\cos (\alpha_1)-\cos (\theta_1))}\left(2 \delta  \sin (\alpha_1)-2 \delta  \sin (\theta_1)+k \sin ^2(\alpha_1)-k \sin ^2(\theta_1)\right.\\[0.5cm]
		& &\left. +2 l \sin (\alpha_1) \cos (\alpha_1) -2 l \sin (\theta_1) \cos (\theta_1)-m \cos ^2(\alpha_1)+m \cos ^2(\theta_1)\right).
	\end{array}$$
	Substituting the solution $(\cos(\alpha_2),\sin(\alpha_2),\cos(\theta_2),\sin(\theta_2))$ and $\varepsilon$  in the second equation of the system we obtain  $\delta=\frac{u}{v}$ being  
	
	$$\begin{array}{lll}	
		u& = & k \cos (\alpha_1) \sin ^2(\alpha_2)-k \sin ^2(\alpha_1) \cos (\alpha_2)-k \cos (\alpha_1) \sin ^2(\theta_2)+k \sin ^2(\alpha_1) \cos (\theta_2)\\[0.5cm]
		& & -k \sin ^2(\alpha_2) \cos (\theta_1)+k \cos (\alpha_2) \sin ^2(\theta_1)-k \sin ^2(\theta_1) \cos (\theta_2)+k \cos (\theta_1) \sin ^2(\theta_2) \\[0.5cm]
		& & -2 l \sin (\alpha_1) \cos (\alpha _1) \cos (\alpha_2)+2 l \cos (\alpha_1) \sin (\alpha_2) \cos (\alpha_2)+2 l \sin (\alpha_1) \cos (\alpha_1) \cos (\theta_2)\\[0.5cm]
		& &-2 l \cos (\alpha_1) \sin (\theta_2) \cos (\theta_2)-2 l \sin (\alpha_2) \cos (\alpha_2) \cos (\theta_1) +2 l \cos (\alpha_2) \sin (\theta_1) \cos (\theta_1) \\[0.5cm]
		& &-2 l \sin (\theta_1) \cos (\theta_1) \cos (\theta_2)+2 l \cos (\theta_1) \sin (\theta_2) \cos (\theta_2)+m \cos ^2(\alpha_1) \cos (\alpha_2) \\[0.5cm]
		& &-m \cos (\alpha_1) \cos ^2(\alpha_2)-m \cos ^2(\alpha_1) \cos (\theta_2)+m \cos (\alpha_1) \cos ^2(\theta_2) -m \cos (\alpha_2) \cos ^2(\theta_1) \\[0.5cm]
		& &+m \cos ^2(\alpha_2) \cos (\theta_1)-m \cos (\theta_1) \cos ^2(\theta_2) +m \cos ^2(\theta_1) \cos (\theta_2), 
	\end{array}$$
	
	$$\begin{array}{lll}
		v&=&2 (\sin (\alpha_1) \cos (\alpha_2)-\cos (\alpha_1) \sin (\alpha_2)-\sin (\alpha_1) \cos (\theta_2)+\cos (\alpha_1) \sin (\theta_2)\\[0.5cm]
		& &+\sin (\alpha_2) \cos (\theta_1)-\cos (\alpha_2) \sin (\theta_1)+\sin (\theta_1) \cos (\theta_2)-\cos (\theta_1) \sin (\theta_2)).
	\end{array}$$
	
	Now substituting the solution $(\cos(\alpha_3),\sin(\alpha_3),\cos(\theta_3),\sin(\theta_3))$, $\varepsilon$ and $\delta$ in the second equation of the system we obtain $l=\frac{t}{w}$ being  
	
	$$\begin{array}{lll}
		t& = &(k+m) \sin \left(\frac{\alpha_1-\theta_1}{2}\right) \sin \left(\frac{\alpha_2-\theta_2}{2}\right) \sin \left(\frac{\alpha_3-\theta_3}{2}\right) \left(-\cos \left(\frac{1}{2} (\alpha_1-3 \alpha_2-\alpha_3+\theta_1-\theta_2-\theta_3)\right)\right. \\[0.5cm]
		& &\left. +\cos \left(\frac{1}{2} (\alpha_1-\alpha_2-3 \alpha_3+\theta_1-\theta_2-\theta_3)\right)-\cos \left(\frac{1}{2} (3\alpha_1+\alpha_2-\alpha_3+\theta_1+\theta_2-\theta_3)\right)\right. \\[0.5cm]
		& &\left. +\cos \left(\frac{1}{2} (\alpha_1+3 \alpha_2-\alpha_3+\theta_1+\theta_2-\theta_3)\right)-\cos \left(\frac{1}{2} (\alpha_1+\alpha_2-\alpha_3+3 \theta_1+\theta_2-\theta_3)\right)\right. \\[0.5cm]
		& &\left. +\cos \left(\frac{1}{2} (\alpha_1+\alpha_2-\alpha_3+\theta_1+3 \theta_2-\theta_3)\right)-\cos \left(\frac{1}{2} (\alpha_1-\alpha_2-\alpha_3+\theta_1-3 \theta_2-\theta_3)\right)\right. \\[0.5cm]
		& &\left. +\cos \left(\frac{1}{2} (3 \alpha_1-\alpha_2+\alpha_3+\theta_1-\theta_2+\theta_3)\right)-\cos \left(\frac{1}{2} (\alpha_1-\alpha_2+3 \alpha_3+\theta_1-\theta_2+\theta_3)\right)\right. \\[0.5cm]
		& &\left. +\cos \left(\frac{1}{2} (\alpha_1-\alpha_2+\alpha_3+3 \theta_1-\theta_2+\theta_3)\right)-\cos \left(\frac{1}{2} (\alpha_1-\alpha_2+\alpha_3+\theta_1-\theta_2+3\theta_3)\right)\right. \\[0.5cm]
		& &\left. +\cos \left(\frac{1}{2} (\alpha_1-\alpha_2-\alpha_3+\theta_1-\theta_2-3 \theta_3)\right)\right),
	\end{array}$$
	
	$$\begin{array}{lll}
		w&=&\cos (\alpha_2) (\cos (\alpha_3) (\sin (\alpha_1)-\sin (\theta_1)) (\sin (\alpha_2)-\sin (\alpha_3))+\cos (\theta_3) (\sin (\alpha_1)-\sin (\theta_1))\\[0.5cm]
		& &(\sin (\theta_3)-\sin (\alpha_2))+\cos (\theta_1) (\sin (\alpha_2)-\sin (\theta_1)) (\sin (\alpha_3)-\sin (\theta_3)))\\[0.5cm]
		& &+\cos (\alpha_1) ((\sin (\alpha_3)-\sin (\theta_3)) (\cos (\alpha_2) (\sin (\alpha_1)-\sin (\alpha_2))+\cos (\theta_2) (\sin (\theta_2)-\sin (\alpha_1)))\\[0.5cm]
		& &+\cos (\alpha_3) (\sin (\alpha_3)-\sin (\alpha_1)) (\sin (\alpha_2)-\sin (\theta_2))+\cos (\theta_3) (\sin (\alpha_1)-\sin (\theta_3)) (\sin (\alpha_2)\\[0.5cm]
		& &-\sin (\theta_2)))+\sin (\alpha_1) \sin (\alpha_3) \cos (\alpha_3) \cos (\theta_2)-\sin (\alpha_1) \cos (\alpha_3) \sin (\theta_2) \cos (\theta_2)\\[0.5cm]
		& &+\sin (\alpha_1) \sin (\theta_2) \cos (\theta_2) \cos (\theta_3)-\sin (\alpha_1) \cos (\theta_2) \sin (\theta_3) \cos (\theta_3)\\[0.5cm]
		& &-\sin (\alpha_2) \sin (\alpha_3) \cos (\alpha_3) \cos (\theta_1)+\sin (\alpha_2) \cos (\alpha_3) \sin (\theta_1) \cos (\theta_1)\\[0.5cm]
		& &-\sin (\alpha_2) \sin (\theta_1) \cos (\theta_1) \cos (\theta_3)+\sin (\alpha_2) \cos (\theta_1) \sin (\theta_3) \cos (\theta_3)\\[0.5cm]
		& &-\sin (\alpha_3) \cos (\alpha_3) \sin (\theta_1) \cos (\theta_2)+\sin (\alpha_3) \sin (\theta_1) \cos (\theta_1) \cos (\theta_2)\\[0.5cm]
		& &+\sin (\alpha_3) \cos (\alpha_3) \cos (\theta_1) \sin (\theta_2)-\sin (\alpha_3) \cos (\theta_1) \sin (\theta_2) \cos (\theta_2)\\[0.5cm]
		& &-\cos (\alpha_3) \sin (\theta_1) \cos (\theta_1) \sin (\theta_2)+\cos (\alpha_3) \sin (\theta_1) \sin (\theta_2) \cos (\theta_2)\\[0.5cm]
		& &+\sin (\theta_1) \cos (\theta_1) \sin (\theta_2) \cos (\theta_3)-\sin (\theta_1) \sin (\theta_2) \cos (\theta_2) \cos (\theta_3)\\[0.5cm]
		& &-\sin (\theta_1) \cos (\theta_1) \cos (\theta_2) \sin (\theta_3)+\sin (\theta_1) \cos (\theta_2) \sin (\theta_3) \cos (\theta_3)\\[0.5cm]
		& &+\cos (\theta_1) \sin (\theta_2) \cos (\theta_2) \sin (\theta_3)-\cos (\theta_1) \sin (\theta_2) \sin (\theta_3) \cos (\theta_3).
	\end{array}$$
	
	Substituting $\varepsilon$, $\delta$ e $l$ in the first integral $H_2$ we have
	$$H_2(x,y)=-\frac{m x^2}{2}+\frac{k y^2}{2}+(k+m)f(x,y).$$
	Similarly we have that
	$$H_1(x,y)=-\frac{c x^2}{2}+\frac{b y^2}{2}+(c+b)f(x,y).$$
	
	Hence, we obtain the system of closing equations below
	$$\left\lbrace \def\arraystretch{1.5} \begin{array}{l}
			-\frac{c p^2}{2}+\frac{b q^2}{2}+(c+b)f(p,q)=-\frac{c r^2}{2}+\frac{b s^2}{2}+(c+b)f(r,s),\\[0.3cm]
			-\frac{m p^2}{2}+\frac{k q^2}{2}+(k+m)f(p,q)=-\frac{m r^2}{2}+\frac{k s^2}{2}+(k+m)f(r,s),\\[0.3cm]
			p^2+q^2=1,\\[0.3cm]
			r^2+s^2=1.
		\end{array}\right.$$
	We have that $q^2=1-p^2$ and $s^2=1-r^2$, hence in the first equation we obtain 
	$$-\frac{c p^2}{2}+\frac{b (1-p^2)}{2}+(c+b)f(p,q)=-\frac{c r^2}{2}+\frac{b (1-r^2)}{2}+(c+b)f(r,s)$$
	$$\Rightarrow -\frac{p^2}{2}(c+b)+\frac{b}{2}+(c+b)f(p,q)=-\frac{r^2}{2}(c+b)+\frac{b}{2}+(c+b)f(r,s)$$
	$$\Rightarrow \tilde{f}(p,q)=\tilde{f}(r,s),$$
	where $$\tilde{f}(x,y)=-\frac{x^2}{2}+f(x,y).$$
	Similarly, substituting $q^2=1-p^2$ and $s^2=1-r^2$ in the second equation we have 
	 $ \tilde{f}(p,q)=\tilde{f}(r,s)$.
	Thus, the system of closing equations becomes
	$$\left\lbrace \def\arraystretch{1.5} \begin{array}{l}
		\tilde{f}(p,q)=\tilde{f}(r,s),\\
		p^2+q^2=1,\\
		r^2+s^2=1.
	\end{array}\right.$$
	We have a polynomial system with three equations and four unknowns, so once we assume that the system admits at least three solutions, we obtain that the system admits infinite solutions. 
	
	Therefore, the maximum number of crossing limit cycles that intersect $S^1$ in two points of the piecewise smooth vector field $Z_{XY}$ is less than or equal to two.
\end{proof}

This theorem admits a generalization, presented below.

\begin{theorem}Consider a piecewise smooth vector field 
	$$Z_{XY}(x,y)=\left\{
	\begin{array}{lcl}
		X(x,y),& e(x,y)\leq 0,\\
		Y(x,y),& e(x,y)\geq 0,
	\end{array}
	\right.$$
	where $X(x,y)=(\eta+ax+by,\zeta+cx+dy)$, $Y(x,y)=(\delta+lx+ky,\varepsilon +mx+ny)$,  $divX=divY=0$ and $e(x,y)=\frac{x^2}{u^2}+\frac{v^2}{b^2}-R$. 
	Then the maximum number of crossing limit cycles of the piecewise smooth vector field $Z_{XY}$ that intersect $S^1$ in two points is less than or equal to two.
\end{theorem}

\begin{proof}
	The idea of the demonstration is similar to the Theorem \ref{twocycle}. 
\end{proof}

\subsection{Saddle-center case}

In this subsection, we prove the Theorem B and we show that there are vector fields in $\mathfrak{X}^{S^1}_2$ that have an infinite number of crossing periodic orbits, vector fields that do not admit periodic orbits, vector fields that have only one crossing limit cycle and vector fields that admit two crossing limit cycles. To do so, we are going to present examples of each case. 

\begin{corollary}
	There is $Z_{XY} \in \mathfrak{X}^{S^1}_2$ that admits an infinite number of crossing periodic orbits that intersect $S^1$ in two points. 
\end{corollary}

\begin{proof}
	Consider the piecewise smooth vector field $Z_{XY}$, where $$X(x,y)=(-6y, 4 - 4x)\quad \mbox{and} \quad Y(x,y)=(-4y, 8 + 8x).$$
	The fields $X$ and $Y$ have the following first integrals 
	$$H_1(x,y)=-4x + 2x^2 - 3y^2\quad \mbox{and}\quad H_2(x,y)=-8x - 4x^2 - 2y^2,$$
	respectively.
	
	To study the existence of crossing limit cycles of the field $Z_{XY}$ we use the closing equations, with which we get the system of nonlinear equations
	$$\left\lbrace \def\arraystretch{1.5} \begin{array}{l}
		-4p + 2p^2 - 3q^2=-4r + 2r^2 - 3s^2,\\
		-8p - 4p^2 - 2s^2=-8r - 4r^2 - 2s^2,\\
		p^2+q^2=1,\\
		r^2+s^2=1,
	\end{array}\right. $$
	on which $(p,q,r,s)$ satisfies $(p,q)\neq (r,s)$. From the closing equations we have $q^2=1-p^2$ and $s^2=1-r^2$, thus obtaining the system
	$$\left\lbrace \def\arraystretch{1.5} \begin{array}{l}
		-4p + 2p^2 - 3(1-p^2)=-4r + 2r^2 - 3(1-p^2),\\
		-8p - 4p^2 - 2(1-p^2)=-8r - 4r^2 - 2(1-r^2),
	\end{array}\right. \Rightarrow 
	\left\lbrace \def\arraystretch{1.5} \begin{array}{l}
		-4p + 5p^2=-4r + 5r^2,\\
		-8p - 2p^2 =-8r - 2r^2,
	\end{array}\right. \Rightarrow$$
	$$ 
	\left\lbrace \def\arraystretch{1.5} \begin{array}{l}
		-4p + 5p^2=-4r + 5r^2,\\
		4p +p^2 =4r +r^2,
	\end{array}\right. \Rightarrow r^2=4p+p^2-4r \Rightarrow 	-4p + 5p^2=-4r + 5(4p+p^2-4r) \Rightarrow p=r.$$
	Thus, we get that $s=\pm q$. In this way, $(p,q,r,s)=(p,q,p,\pm q)$, and since $(p,q)\neq (r,s)$ we cannot have $s=q$. Hence, $(p,q,r,s)=(p,q,p,- q)$. 
	
	Therefore, $Z_{XY}$ admits an infinite number of crossing periodic orbits that intersect $S^1$ in two points, which we illustrate in Figure \ref{scorbper}.
	
\end{proof}

\begin{corollary}
	There is $Z_{XY} \in \mathfrak{X}^{S^1}_2$ that does not have crossing periodic orbits that intersect $S^1$ in two points.
\end{corollary}

\begin{proof}
	Consider the piecewise smooth vector field $Z_{XY}$, on which $$X(x,y)=(-4 + 2 x - 4 y, 4 - 2 x - 2 y)\quad \mbox{and} \quad Y(x,y)=(-y, 1 + x).$$
	The fields $X$ and $Y$ have the following first integrals
	$$H_1(x,y)=-4 x + x^2 - 4 y + 2 x y - 2 y^2\quad \mbox{and}\quad H_2(x,y)=-x-\frac{x^2}{2}-\frac{y^2}{2},$$
	respectively.
	
	To study the existence of crossing limit cycles of the field $Z_{XY}$ we use the closing equations, with which we get the system of nonlinear equations
	$$\left\lbrace \def\arraystretch{1.5} \begin{array}{l}
		-4p + p^2 - 4 q + 2pq - 2q^2=-4r + r^2 - 4s + 2rs - 2r^2,\\
		p-\frac{p^2}{2}-\frac{q^2}{2}=r-\frac{r^2}{2}-\frac{s^2}{2},\\
		p^2+q^2=1,\\
		r^2+s^2=1,
	\end{array}\right. $$
	where $(p,q,r,s)$ satisfies $(p,q)\neq (r,s)$. From the closing equations we have $q^2=1-p^2$ and $s^2=1-r^2$, thus obtaining the system 
	$$\left\lbrace \def\arraystretch{1.5} \begin{array}{l}
		-4p + p^2 - 4 q + 2pq - 2(1-p^2)=-4r + r^2 - 4s + 2rs - 2(1-r^2),\\
		p-\frac{p^2}{2}-\frac{1-p^2}{2}=r-\frac{r^2}{2}-\frac{1-r^2}{2},
	\end{array}\right. $$
	with which we get that $p=r$. Substituting this in the first equation of the system we have
	$$-4p + p^2 - 4 q + 2pq - 2(1-p^2)=-4p + p^2 - 4s + 2ps - 2(1-p^2) $$
	$$\Rightarrow q(-4+2p)=s(-4+2p) \Rightarrow r=s.$$
	That way, $(p,q,r,s)=(p,q,p,q)$. 
	
	Therefore, since $(p,q) \neq (r,s)$, $Z_{XY}$ does not have crossing periodic orbits that intersect $S^1$ in two points, as we illustrate in Figure \ref{scsemorb}.
	
\end{proof}

\begin{corollary}
	There is $Z_{XY} \in \mathfrak{X}^{S^1}_2$ that has exactly one crossing limit cycle that intersects $S^1$ in two points.
\end{corollary}

\begin{proof}
	Consider the piecewise smooth vector field $Z_{XY}$, on which $$X(x,y)=\left(4y-\frac{9}{2},-2x-\frac{4}{5}\right)\quad \mbox{and} \quad Y(x,y)=\left(-2 y+\frac{1}{3},-4 x-\frac{3}{2}\right).$$
	The fields $X$ and $Y$ have the following first integrals
	$$H_1(x,y)=2 x^2+\frac{3 x}{2}-y^2+\frac{y}{3}\quad \mbox{and} \quad H_2(x,y)=\frac{4}{5} x+x^2+2 y^2-\frac{9 y}{2},$$
	respectively.
	
	To study the existence of crossing limit cycles of the field $Z_{XY}$ we use the closing equations, with which we get the system of nonlinear equations
	$$\left\lbrace \def\arraystretch{1.5} \begin{array}{l}
		2 p^2+\frac{3 p}{2}-q^2+\frac{q}{3}=2 r^2+\frac{3 r}{2}-s^2+\frac{s}{3},\\
		\frac{4}{5} p+p^2+2 q^2-\frac{9 q}{2}=\frac{4}{5} r+r^2+2 s^2-\frac{9 s}{2},\\
		p^2+q^2=1,\\
		r^2+s^2=1,
	\end{array}\right. $$
	where $(p,q,r,s)$ satisfies $(p,q)\neq (r,s)$. To solve this system, we used the Grobner base, obtaining the solutions
	$$\left(-\frac{\sqrt{\frac{2046470263}{84857}}}{468}-\frac{421}{1580}, \frac{421}{468}-\frac{\sqrt{\frac{2046470263}{84857}}}{1580}, \frac{\sqrt{\frac{2046470263}{84857}}}{468}-\frac{421}{1580}, \frac{\sqrt{\frac{2046470263}{84857}}}{1580}+\frac{421}{468} \right), $$
	$$\left(\frac{\sqrt{\frac{2046470263}{84857}}}{468}-\frac{421}{1580}, \frac{\sqrt{\frac{2046470263}{84857}}}{1580}+\frac{421}{468},-\frac{\sqrt{\frac{2046470263}{84857}}}{468}-\frac{421}{1580}, \frac{421}{468}-\frac{\sqrt{\frac{2046470263}{84857}}}{1580} \right), $$
	that represent the same closed curve. 
	
	To complete, we check if these points are in the crossing region. Remember that $h(x,y)=x^2+y^2-1$, then $\nabla h=(2x,2y)$. Hence,
	$$Xh(x,y)=\left\langle  \left(4y-\frac{9}{2},-2x-\frac{4}{5}\right), (2x,2y) \right\rangle = \frac{2 x}{3}-3 y-12 x y,$$
	$$Yh(x,y)=\left\langle  \left(-2 y+\frac{1}{3},-4 x-\frac{3}{2}\right), (2x,2y) \right\rangle =-9 x-\frac{8 y}{5}+4 x y.$$  
	Thus, 
	$$Xh(p,q)= \frac{61273961 \sqrt{173657327107391}-94577623204545}{241653074633100}>0,$$ 
	$$Yh(p,q) = \frac{113022317 \sqrt{173657327107391}+94577623204545}{724959223899300}>0,$$
	$$Xh(r,s) =\frac{24714755 \sqrt{173657327107391}-374942037011263}{241653074633100}<0,$$ 
	$$Yh(r,s) = \frac{2572781848965439-232415235 \sqrt{173657327107391}}{724959223899300}<0,$$
	which imply $Xh(p,q)Yh(p,q)>0$ and $Xh(r,s)Yh(r,s)>0$, that is, $(p,q)$ and $(r,s)$ belong to the crossing region of the discontinuity manifold $S^1$.
	
	Therefore, $Z_{XY}$ has only one crossing limit cycles that intersect $S^1$ in two points, which we illustrate in Figure \ref{sc1ciclo} .
	
\end{proof}	

\begin{corollary}\label{sc2cycle}
	There is $Z_{XY} \in \mathfrak{X}^{S^1}_2$ that has exactly two crossing limit cycle that intersects $S^1$ in two points.
\end{corollary}

\begin{proof}
	
	Consider the piecewise smooth vector field $Z_{XY}$, on which $$X(x,y)=\left(-\frac{2}{3}+\frac{2 x}{3}-\frac{5 y}{3},4-4 x-\frac{2 y}{3}\right)\quad \mbox{and} \quad Y(x,y)=\left(-\frac{2}{3}-\frac{2 x}{3}-\frac{y}{2},4+4 x+\frac{2 y}{3}\right).$$
	The fields $X$ and $Y$ have the following first integrals  
	$$H_1(x,y)= -4x-\frac{2 y}{3} +\frac{2 x y}{3} + 2 x^2-\frac{5 y^2}{6}\quad \mbox{and} \quad H_2(x,y)=-4 x -\frac{2 y}{3} -\frac{2 x y}{3}-2 x^2-\frac{y^2}{4},$$
	respectively.
	
	To study the existence of crossing limit cycles of the field $Z_{XY}$ we use the closing equations, with which we obtain the system of nonlinear equations
	$$\left\lbrace \def\arraystretch{1.5} \begin{array}{l}
		-4p-\frac{2 q}{3} +\frac{2 p q}{3} + 2 p^2-\frac{5 q^2}{6}=-4r-\frac{2 s}{3} +\frac{2 r s}{3} + 2 r^2-\frac{5 s^2}{6},\\
		-4 p -\frac{2 q}{3} -\frac{2 p q}{3}-2 p^2-\frac{q^2}{4}=-4 r -\frac{2 s}{3} -\frac{2 r s}{3}-2 r^2-\frac{s^2}{4},\\
		p^2+q^2=1,\\
		r^2+s^2=1,
	\end{array}\right. $$
	where $(p,q,r,s)$ satisfies $(p,q)\neq (r,s)$. To solve this system we used the Grobner base, obtaining the solutions  $(p_1,q_1,r_1,s_1)$, $(r_1,s_1,p_1,q_1)$, $(p_2,q_2,r_2,s_2)$ and $(r_2,s_2,q_2,s_2)$, where
	$$p_1=-\frac{1}{\sqrt{37}},\;q_1=\frac{6}{\sqrt{37}},\;r_1=\frac{1}{\sqrt{37}},\;s_1=-\frac{6}{\sqrt{37}},$$
	
	$$p_2= \frac{1}{26} \left(-\sqrt{3281}-\sqrt{\frac{13604122}{\sqrt{3281}}-237494}+41\right),$$
	$$q_2=\frac{1}{416} \left(96 \sqrt{3281}+\sqrt{13604122 \sqrt{3281}-779217814}-5536\right)+\frac{55}{208} \sqrt{\frac{6802061}{2 \sqrt{3281}}-\frac{118747}{2}},$$
	$$r_2=\frac{1}{26} \left(41-\sqrt{3281}\right)+\frac{1}{13} \sqrt{\frac{6802061}{2 \sqrt{3281}}-\frac{118747}{2}},$$
	$$s_2=\frac{1}{416} \left(96 \sqrt{3281}-\sqrt{13604122 \sqrt{3281}-779217814}-5536\right)-\frac{55}{208} \sqrt{\frac{6802061}{2 \sqrt{3281}}-\frac{118747}{2}},$$
	that define two closed curves. 
	
	To complete, we check if these points are in the crossing region. Remember that $h(x,y)=x^2+y^2-1$, then $\nabla h=(2x,2y)$. So,
	$$Xh(x,y)=\left\langle  \left(-\frac{2}{3}+\frac{2 x}{3}-\frac{5 y}{3},4-4 x-\frac{2 y}{3}\right), (2x,2y) \right\rangle = \frac{2}{3} \left(2 x^2-x (17 y+2)-2 (y-6) y\right),$$
	$$Yh(x,y)=\left\langle  \left(-\frac{2}{3}-\frac{2 x}{3}-\frac{y}{2},4+4 x+\frac{2 y}{3}\right), (2x,2y) \right\rangle =\frac{1}{3} \left(-4 x^2+x (21 y-4)+4 y (y+6)\right).$$  
	Thus, 
	$$Xh(p_1,q_1)= \frac{4 \sqrt{37}}{3}+\frac{64}{111}>0,\quad Yh(p_1,q_1) = \frac{2}{111} \left(74 \sqrt{37}+7\right)>0,$$
	$$Xh(r_1,s_1) =\frac{64}{111}-\frac{4 \sqrt{37}}{3}<0, \quad Yh(r_1,s_1) = \frac{14}{111}-\frac{4 \sqrt{37}}{3}<0,$$
	{\small $$Xh(p_2,q_2)= \frac{385 \sqrt{\frac{1}{2} \left(6802061 \sqrt{3281}-389608907\right)}}{1352}+\frac{4234}{39}-\frac{31963 \sqrt{\frac{6802061}{2 \sqrt{3281}}-\frac{118747}{2}}}{4056}-\frac{80706}{13 \sqrt{3281}}>0,$$ 
		$$Yh(p_2,q_2) = -\frac{971 \sqrt{\frac{1}{2} \left(6802061 \sqrt{3281}-389608907\right)}}{8112}+\frac{70595 \sqrt{\frac{6802061}{2 \sqrt{3281}}-\frac{118747}{2}}}{8112}+\frac{4234}{39}-\frac{80706}{13 \sqrt{3281}}>0,$$
		$$Xh(r_2,s_2) =-\frac{385 \sqrt{\frac{1}{2} \left(6802061 \sqrt{3281}-389608907\right)}}{1352}+\frac{31963 \sqrt{\frac{6802061}{2 \sqrt{3281}}-\frac{118747}{2}}}{4056}+\frac{4234}{39}-\frac{80706}{13 \sqrt{3281}}<0,$$ 
		$$Yh(r_2,s_2) = \frac{971 \sqrt{\frac{1}{2} \left(6802061 \sqrt{3281}-389608907\right)}}{8112}+\frac{4234}{39}-\frac{70595 \sqrt{\frac{6802061}{2 \sqrt{3281}}-\frac{118747}{2}}}{8112}-\frac{80706}{13 \sqrt{3281}}<0,$$}
	which imply 
	$$Xh(p_1,q_1)Yh(p_1,q_1)>0,\;\; Xh(r_1,s_1)Yh(r_1,s_1)>0,$$
	$$ Xh(p_2,q_2)Yh(p_2,q_2)>0,\;\; Xh(r_2,s_2)Yh(r_2,s_2)>0,$$ 
	that is, the points belong to the crossing region of the discontinuity manifold $S^1$.
	
	Therefore, $Z_{XY}$ has two crossing limit cycles that intersect $S^1$ in two points, as we illustrate in Figure \ref{sc2ciclos} .
\end{proof}

The next result shows the existence of a homoclinic orbit on vector fields belonging to $\mathfrak{X}^{S^1}_1$.

\begin{corollary}\label{selacentrociclohet}
	There is $Z_{XY} \in \mathfrak{X}^{S^1}_2$ that admits exactly one crossing limit cycle and one homoclinic orbit that intersect $S^1$ in two points.
\end{corollary}

\begin{proof}
	Consider the piecewise smooth vector field $Z_{XY}$, on which $$X(x,y)=(x-4 y-1,-4 x-y+4)\quad \mbox{and} \quad Y(x,y)=(-x-4 y-1,4 x+y+4).$$
	The fields $X$ and $Y$ have the following first integrals 
	$$H_1(x,y)=2 x^2+x y-4 x-2 y^2-y\quad \mbox{and}\quad H_2(x,y)=-2 x^2-x y-4 x-2 y^2-y,$$
	respectively.
	
	To study the existence of crossing limit cycles of the field $Z_{XY}$ we use the closing equations, with which we get the system of nonlinear equations
	$$\left\lbrace \def\arraystretch{1.5} \begin{array}{l}
		2 p^2+p q-4 p-2 q^2-q=2 r^2+r s-4 r-2 s^2-s,\\
		-2 p^2-p q-4 p-2 q^2-q=-2 r^2-r s-4 r-2 s^2-s,\\
		p^2+q^2=1,\\
		r^2+s^2=1,
	\end{array}\right. $$
	where $(p,q,r,s)$ satisfies $(p,q)\neq (r,s)$. To solve this system we used the Grobner base, obtaining the solutions $(p_1,q_1,r_1,s_1)$, $(r_1,s_1,p_1,q_1)$, $(p_2,q_2,r_2,s_2)$ and $(r_2,s_2,p_2,q_2)$, where
	$$p_1=-\frac{1}{\sqrt{17}},\;q_1=\frac{4}{\sqrt{17}},\;r_1=\frac{1}{\sqrt{17}},\;s_1-\frac{4}{\sqrt{17}},$$
	
	$$p_2= \frac{1}{4} \left(-\sqrt{5}-\sqrt{\frac{344}{\sqrt{5}}-153}+2\right),\; q_2=\sqrt{5}+\frac{1}{2} \sqrt{\frac{344}{\sqrt{5}}-153}+\frac{1}{4} \sqrt{344 \sqrt{5}-765}-\frac{9}{4},$$
	$$r_2=\frac{1}{20} \left(-5 \sqrt{5}+\sqrt{5 \left(344 \sqrt{5}-765\right)}+10\right),\; s_2=\sqrt{5}-\frac{1}{2} \sqrt{\frac{344}{\sqrt{5}}-153}-\frac{1}{4} \sqrt{344 \sqrt{5}-765}-\frac{9}{4},$$
	that represent two closed curves. 
	
	To complete, we check if these points are in the crossing region. Remember that $h(x,y)=x^2+y^2-1$, then $\nabla h=(2x,2y)$. Thus,
	$$Xh(x,y)=\left\langle  (x-4 y-1,-4 x-y+4), (2x,2y) \right\rangle = \frac{2}{3} \left(2 x^2-x (17 y+2)-2 (y-6) y\right),$$
	$$Yh(x,y)=\left\langle  (-x-4 y-1,4 x+y+4), (2x,2y) \right\rangle =-2 \left(x^2+x-y (y+4)\right).$$  
	Hence, 
	$$Xh(p_1,q_1)= 2 \left(\sqrt{17}+1\right)>0,\quad Yh(p_1,q_1) = 2 \sqrt{17}+\frac{30}{17}>0,$$
	$$Xh(r_1,s_1) =2-2 \sqrt{17}<0, \quad Yh(r_1,s_1) = \frac{30}{17}-2 \sqrt{17}<0,$$
	{\small $$Xh(p_2,q_2)= \frac{1}{10} \left(-77 \sqrt{5}+65 \sqrt{344 \sqrt{5}-765}-9 \sqrt{5 \left(344 \sqrt{5}-765\right)}+190\right)>0,$$ 
		$$Yh(p_2,q_2) = \frac{1}{10} \left(-77 \sqrt{5}+15 \sqrt{344 \sqrt{5}-765}+11 \sqrt{5 \left(344 \sqrt{5}-765\right)}+190\right)>0,$$
		$$Xh(r_2,s_2) =\frac{1}{10} \left(-77 \sqrt{5}-65 \sqrt{344 \sqrt{5}-765}+9 \sqrt{5 \left(344 \sqrt{5}-765\right)}+190\right)<0,$$ 
		$$Yh(r_2,s_2) = \frac{1}{10} \left(-77 \sqrt{5}-15 \sqrt{344 \sqrt{5}-765}-11 \sqrt{5 \left(344 \sqrt{5}-765\right)}+190\right)<0,$$}
	which imply 
	$$Xh(p_1,q_1)Yh(p_1,q_1)>0,\;\; Xh(r_1,s_1)Yh(r_1,s_1)>0,$$
	$$ Xh(p_2,q_2)Yh(p_2,q_2)>0,\;\; Xh(r_2,s_2)Yh(r_2,s_2)>0,$$ 
	that is, the points belong to the crossing region of the discontinuity manifold $S^1$.
	
	Now, we are going to verify that the closed curve passing through $(p_1,q_1)$ and $(r_1,s_1)$ is a homoclinic orbit. We know that the field $X$ is a saddle and the level curves of the function $H_1(x,y)$ outline the phase portrait of the field. Thus, evaluating $H_1$ in the singularity $(1,0)$, we have $H_1(1,0)=-2$, then, the curve $H_1(x,y)=-2$ represents the stable manifold and the unstable manifold of the field $X$. We note that the points $(p_1,q_1)$ and $(r_1,s_1)$ belong to the curve $H_1(x,y)=-2$, because $H_1(p_1,q_1)=-2$ and $H_1(r_1,s_1)=-2$. Hence, we conclude the existence of a homoclinic orbit. 
	
	Therefore, $Z_{XY}$ has one crossing limit cycle and one homoclinic orbit that intersect $S^1$ in two points, as we illustrate in Figure \ref{selacentrohet}.
	
\end{proof}	

\begin{figure}[h]
	\begin{subfigure}{0.33\textwidth}
		\centering
		\includegraphics[scale=0.3]{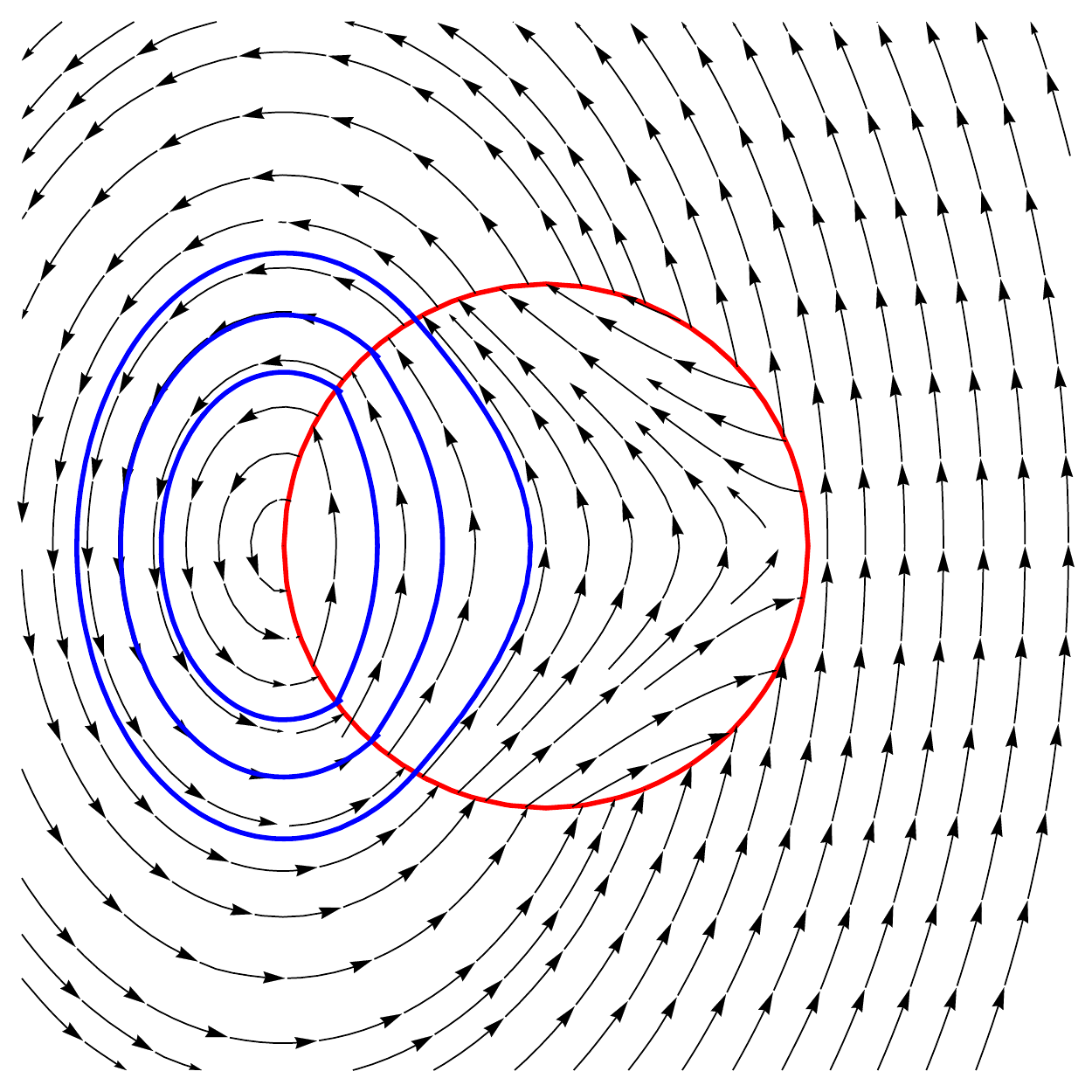}
		\caption{Periodic orbits}
		\label{scorbper}
	\end{subfigure}%
	\begin{subfigure}{0.33\textwidth}
		\centering
		\includegraphics[scale=0.3]{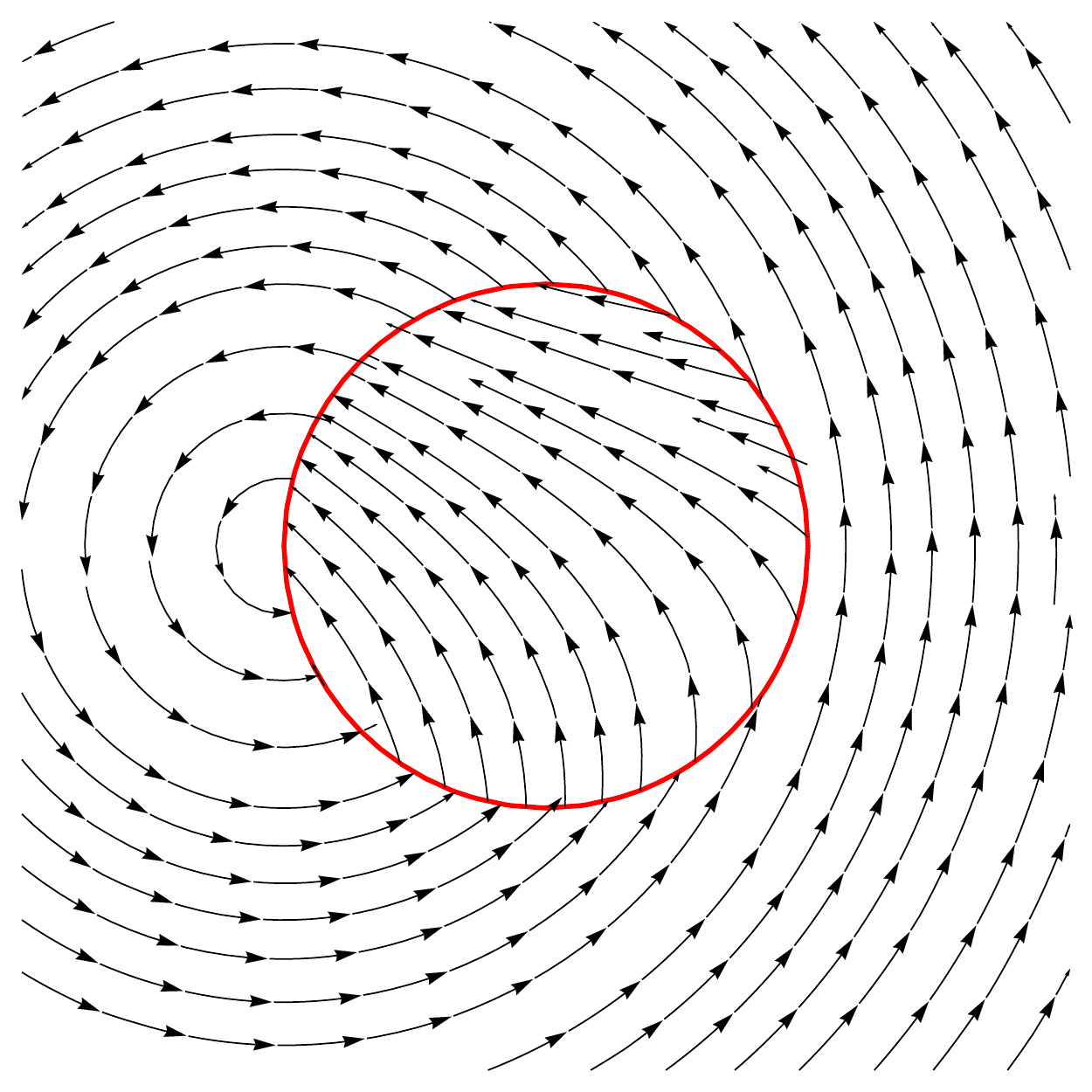}
		\caption{Without periodic orbits}
		\label{scsemorb}
	\end{subfigure}%
	\begin{subfigure}{0.33\textwidth}
		\centering
		\includegraphics[scale=0.3]{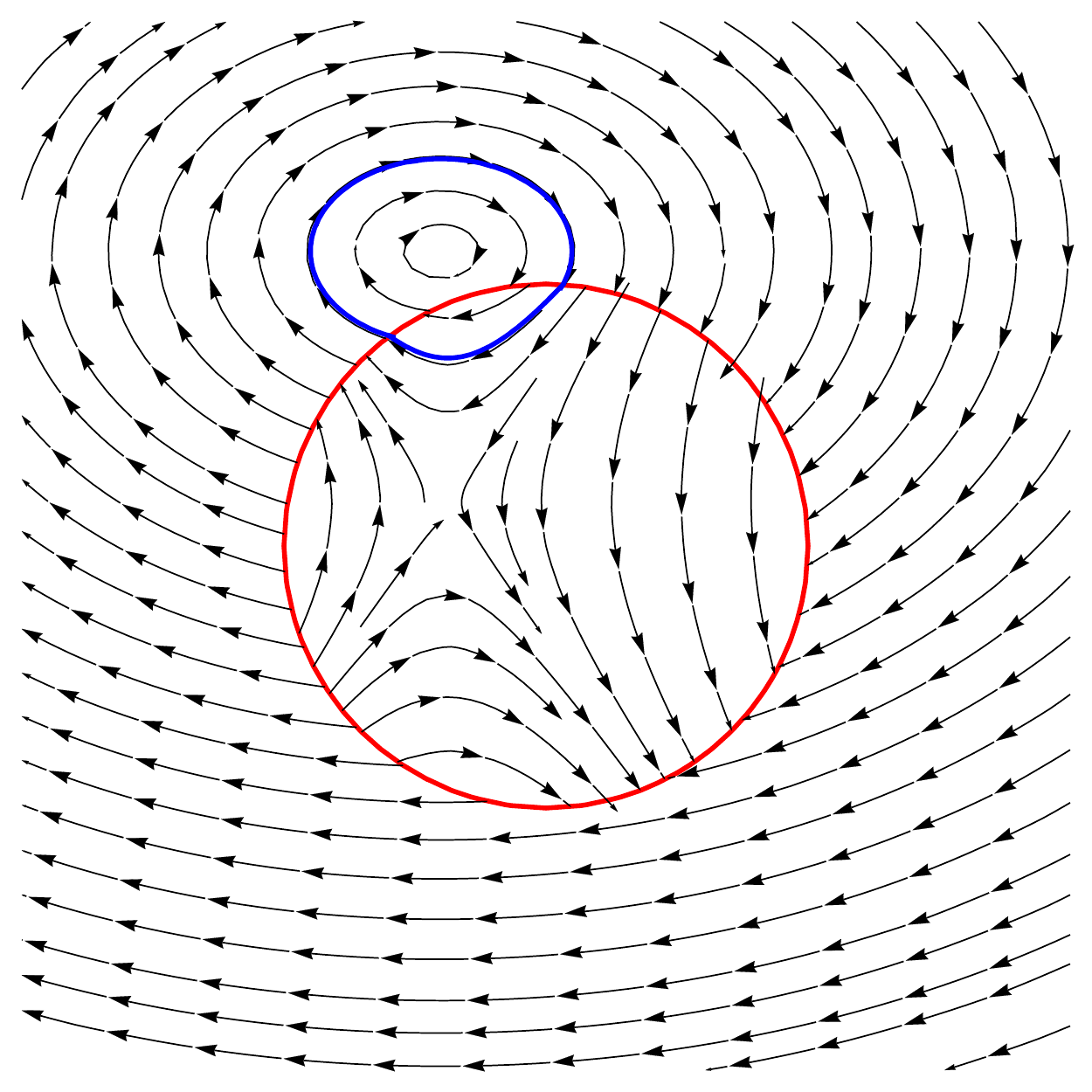}
		\caption{ One limit cycles}
		\label{sc1ciclo}
	\end{subfigure}%
	\\
	\begin{subfigure}{0.33\textwidth}
		\centering
		\includegraphics[scale=0.3]{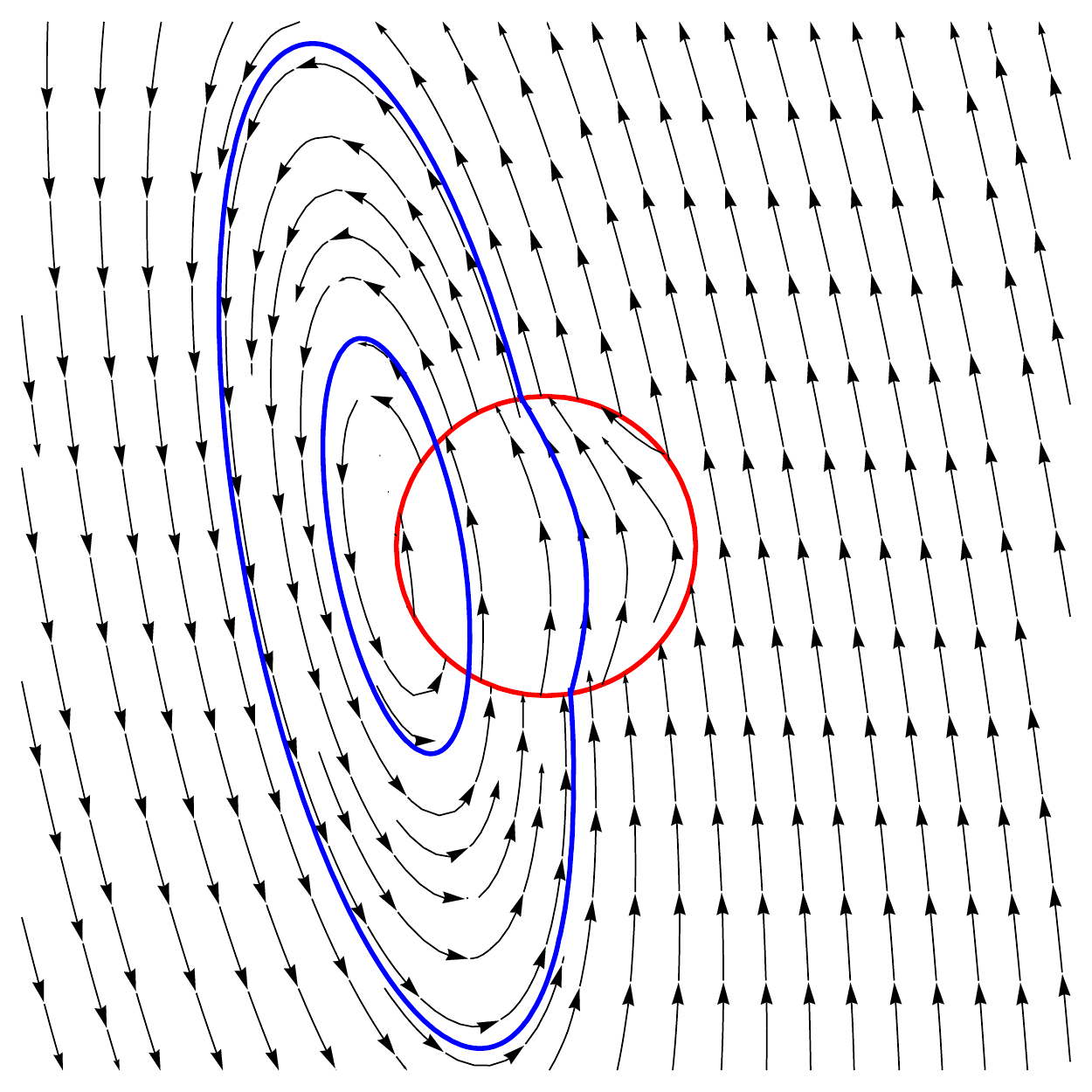}
		\caption{Two-limit cycles}
		\label{sc2ciclos}
	\end{subfigure}%
	\begin{subfigure}{0.33\textwidth}
		\centering
		\includegraphics[scale=0.3]{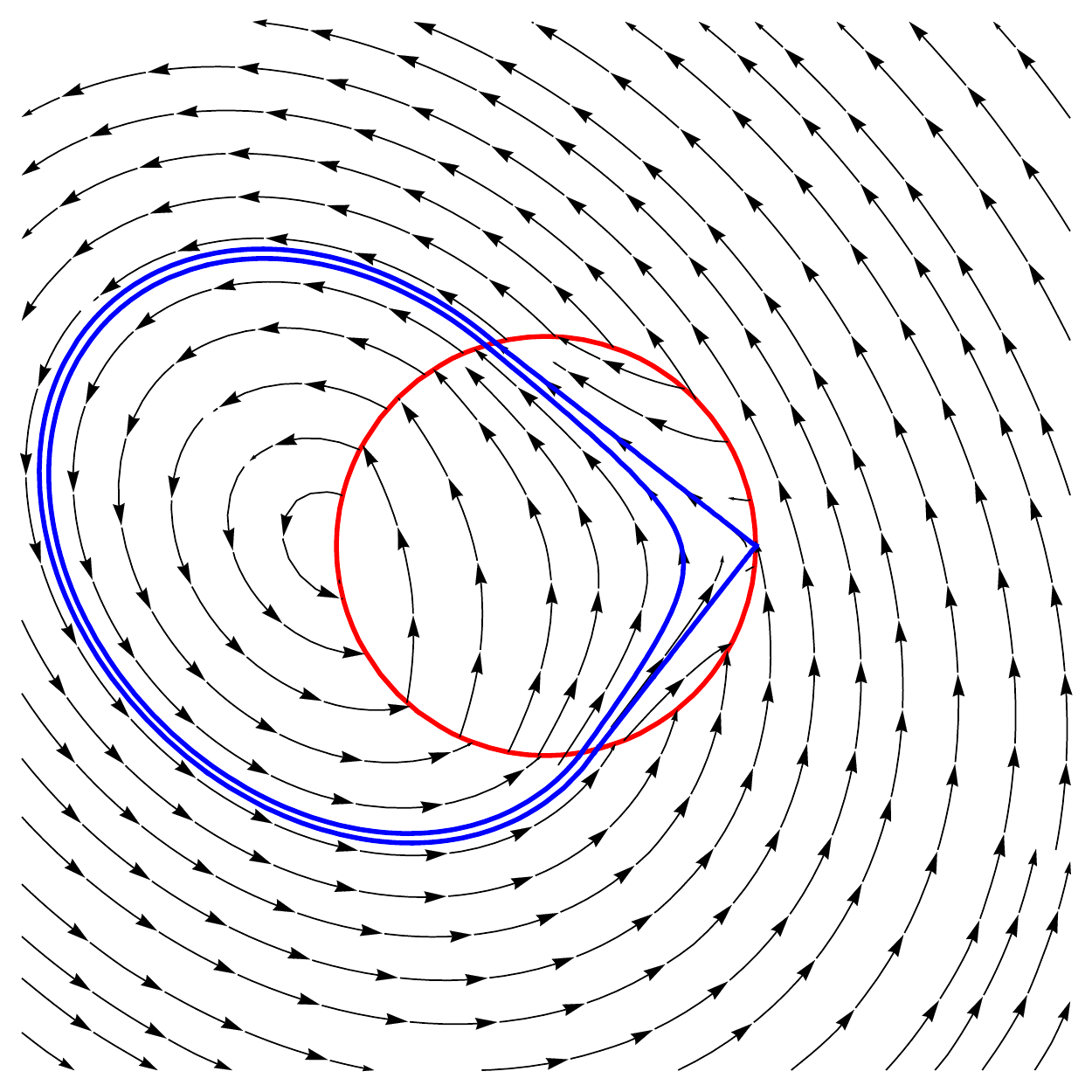}
		\caption{Homoclinic orbit}
		\label{selacentrohet}
	\end{subfigure}%
	\caption{Phase portrait of vector fields $Z_{XY}\in \mathfrak{X}^{S^1}_2$.}
	\label{caso1}
\end{figure} 

\begin{proof}[Proof of Theorem B]
	From the Theorem \ref{twocycle} and Corollary \ref{sc2cycle} we have that piecewice smooth vector fields in $\mathfrak{X}^{S^1}_2$ admit at most two crossing limit cycles that intersect $S^1$ in two points. 
\end{proof}

\subsection{Center-saddle case}

In this subsection, we show that there are vector fields in $\mathfrak{X}^{S^1}_3$ that admit an infinite number of crossing periodic orbits, vector fields that do not admit periodic orbits and vector fields that have only one crossing limit cycle. To do so, we are going to present examples of each case.

Firstly, we note that by the Theorem \ref{twocycle} the maximum number of crossing limit cycles that intersect $S^1$ in two points of $Z_{XY}\in \mathfrak{X}^{S^1}_3$ is less than or equal to two, that is, the  Theorem C holds for this case. 

\begin{corollary}
	There is $Z_{XY} \in \mathfrak{X}^{S^1}_3$ that admits an infinite number of crossing periodic orbits that intersect $S^1$ in two points. 
\end{corollary}

\begin{proof}
	
	Consider the piecewise smooth vector field $Z_{XY}$, on which
	$$X(x,y)=(-6+2y, -14x)\quad \mbox{and} \quad Y(x,y)=(7-2y, -18x).$$
	The fields $X$ and $Y$ have the following first integrals
	$$H_1(x,y)=-6y + 7x^2 +y^2\quad \mbox{and}\quad H_2(x,y)=7y +9x^2 -y^2,$$
	respectively. 
	
	To study the existence of crossing limit cycles of the field $Z_{XY}$ we use the closing equations, with which we obtain the system of nonlinear equations
	$$\left\lbrace \def\arraystretch{1.5} \begin{array}{l}
		-6q + 7p^2 +q^2=-6s + 7r^2 +s^2,\\
		7q +9p^2 -q^2=7s +9r^2 -s^2,\\
		p^2+q^2=1,\\
		r^2+s^2=1,
	\end{array}\right. $$
	where $(p,q,r,s)$ satisfies $(p,q)\neq (r,s)$. From the closing equations we have that $q^2=1-p^2$ and $s^2=1-r^2$, thus getting the system  
	$$\left\lbrace \def\arraystretch{1.5} \begin{array}{l}
		-6q + 7p^2 +1-p^2=-6s + 7r^2 +1-r^2,\\
		7q +9p^2 -1+p^2=7s +9r^2 -1+s^2,
	\end{array}\right. \Rightarrow 
	\left\lbrace \def\arraystretch{1.5} \begin{array}{l}
		-6q + 6p^2=-6s + 6r^2,\\
		7q +10p^2=7s +10r^2,
	\end{array}\right. \Rightarrow$$
	$$ 
	\left\lbrace \def\arraystretch{1.5} \begin{array}{l}
		-q + p^2=-s + r^2,\\
		7q +10p^2=7s +10r^2,
	\end{array}\right. \Rightarrow \left\lbrace \def\arraystretch{1.5} \begin{array}{l}
		-7q + 7p^2=-7s + 7r^2,\\
		7q +10p^2=7s +10r^2,
	\end{array}\right. \Rightarrow 17p^2=17r^2 \Rightarrow r=\pm p.$$
	So, substituting the obtained information on the first equation of the initial system, we get  
	$$-6q + 7p^2 +q^2=-6s + 7r^2 +s^2 \Rightarrow -6q+7p^2+1-p^2=-6s+7p^2+1-p^2 \Rightarrow q=s.$$
	Thus, $(p,q,r,s)=(p,q,\pm p,q)$, and since $(p,q)\neq (r,s)$ we cannot have $r=p$. Hence, $(p,q,r,s)=(p,q,-p, q)$.
	
	Therefore, $Z_{XY}$ admits an infinite number of crossing periodic orbits that intersect $S^1$ in two points, which we illustrate in Figure \ref{scconjorb}.
	
\end{proof}

\begin{corollary}
	There is $Z_{XY} \in \mathfrak{X}^{S^1}_3$ that does not have crossing periodic orbits that intersect $S^1$ in two points. 
\end{corollary}

\begin{proof}
	
	Consider the piecewise smooth vector field $Z_{XY}$, where 
	$$X(x,y)=(2-2 y,2 x)\quad \mbox{and} \quad Y(x,y)=(-2 y-2,2-2 x).$$
	The fields $X$ and $Y$ have the following first integrals
	$$H_1(x,y)=2y-x^2-y^2\quad \mbox{and}\quad H_2(x,y)=-2x-2y+x^2-y^2,$$
	respectively.
	
	To study the existence of crossing limit cycles of the field $Z_{XY}$ we use the closing equations, with which we get the system of nonlinear equations 
	$$\left\lbrace \def\arraystretch{1.5} \begin{array}{l}
		2q - p^2 - q^2 = 2s - r^2 - s^2,\\
		-2p -2q + p^2 - q^2 = -2r - 2s + r^2 - s^2,\\
		p^2+q^2=1,\\
		r^2+s^2=1,
	\end{array}\right. $$
	on which $(p,q,r,s)$ satisfies $(p,q)\neq (r,s)$. From the closing equations we have 
	$$2p - p^2 - q^2 = 2s - r^2 - s^2 \Rightarrow 2q-1=2s-1 \Rightarrow q=s.$$
	That away, with $p^2=1-q^2$ and $r^2=1-s^2$ we get $r^2=p^2$. Substituting this information on the second line of the system, we have
	$$-2p -2q + p^2 - q^2 = -2r - 2s + r^2 - s^2 \Rightarrow -2p -2q + p^2 - q^2 = -2r - 2q + p^2 - q^2 \Rightarrow p=r.$$
	Thus, $(p,q,r,s)=(p,q,p,q)$. 
	
	Therefore, since $(p,q)\neq (r,s)$, $Z_{XY}$ does not admit crossing periodic orbits that intersect $S^1$ in two points, which we illustrate in Figure \ref{centroselasemorb}.
	
\end{proof}	

\begin{corollary}
	There is $Z_{XY} \in \mathfrak{X}^{S^1}_3$ that has exactly one crossing limit cycle that intersects $S^1$ in two points.
\end{corollary}

\begin{proof}
	
	Consider the piecewise smooth vector field $Z_{XY}$, on which $$X(x,y)=\left(5 - 6 y, \frac{1}{2} + 8 x \right)\quad \mbox{and} \quad Y(x,y)=\left(-10 + 2y, \frac{1}{5} + 20x \right).$$
	The fields $X$ and $Y$ admit the following first integrals
	$$H_1(x,y)=-\frac{x}{2}+5y - 4x^2 - 3y^2 \quad \mbox{and} \quad H_2(x,y)=-\frac{x}{5}- 10 y - 10x^2 + y^2,$$
	respectively.
	
	To study the existence of crossing limit cycles of the field $Z_{XY}$ we use the closing equations, with which we obtain the system of nonlinear equations 
	$$\left\lbrace \def\arraystretch{1.5} \begin{array}{l}
		-\frac{p}{2}+ 5q-4p^2-3q^2 =-\frac{r}{2}+ 5s-4r^2-3s^2 , \\
		-\frac{p}{5}-10 q-10p^2+q^2=-\frac{r}{5}-10s-10r^2+s^2,\\
		p^2+q^2=1,\\
		r^2+s^2=1,
	\end{array}\right. $$
	where $(p,q,r,s)$ satisfies $(p,q)\neq (r,s)$. To solve this system we used the Grobner base, obtaining the solutions
	$$\left(-\frac{20\sqrt{\frac{2010061}{425309}}}{53}-\frac{3}{65}, \frac{30}{53}-\frac{2 \sqrt{\frac{2010061}{425309}}}{65}, \frac{20 \sqrt{\frac{2010061}{425309}}}{53}-\frac{3}{65}, \frac{2 \sqrt{\frac{2010061}{425309}}}{65}+\frac{30}{53}\right), $$
	$$\left(\frac{20 \sqrt{\frac{2010061}{425309}}}{53}-\frac{3}{65}, \frac{2 \sqrt{\frac{2010061}{425309}}}{65}+\frac{30}{53},-\frac{20\sqrt{\frac{2010061}{425309}}}{53}-\frac{3}{65}, \frac{30}{53}-\frac{2 \sqrt{\frac{2010061}{425309}}}{65}\right), $$
	that define the same closed curve. 
	
	To complete, we check if these points are in the crossing region. Remember that $h(x,y)=x^2+y^2-1$, so $\nabla h=(2x,2y)$. Hence,
	$$Xh(x,y)=\left\langle  \left(5 - 6 y, \frac{1}{2} + 8 x \right), (2x,2y) \right\rangle =10x + y + 4 xy,$$
	$$Yh(x,y)=\left\langle  \left(-10 + 2y, \frac{1}{5} + 20x \right), (2x,2y) \right\rangle =-20x + \frac{2y}{5} + 44xy.$$  
	Thus, 
	$$Xh(p,q)= \frac{2 \left(553972811600-27611377 \sqrt{854897033849}\right)}{5047577844725}<0,$$
	$$ Yh(p,q) =-\frac{4 \left(5343623 \sqrt{854897033849}-3046850463800\right)}{5047577844725}<0,$$
	$$Xh(r,s) =\frac{2 \left(27611377 \sqrt{854897033849}+553972811600\right)}{5047577844725}>0,$$
	$$ Yh(r,s) = \frac{4 \left(5343623 \sqrt{854897033849}+3046850463800\right)}{5047577844725}>0,$$
	which imply $Xh(p,q)Yh(p,q)>0$ and $Xh(r,s)Yh(r,s)>0$, that is, $(p,q)$ and $(r,s)$ belong to the crossing region of the discontinuity manifold $S^1$.
	
	Therefore, $Z_{XY}$ has only one crossing limit cycle that intersects $S^1$ in two points, which we illustrate in Figure \ref{selacentro1ciclo}.
\end{proof}

\begin{figure}[h]
	\begin{subfigure}{0.33\textwidth}
		\centering
		\includegraphics[scale=0.3]{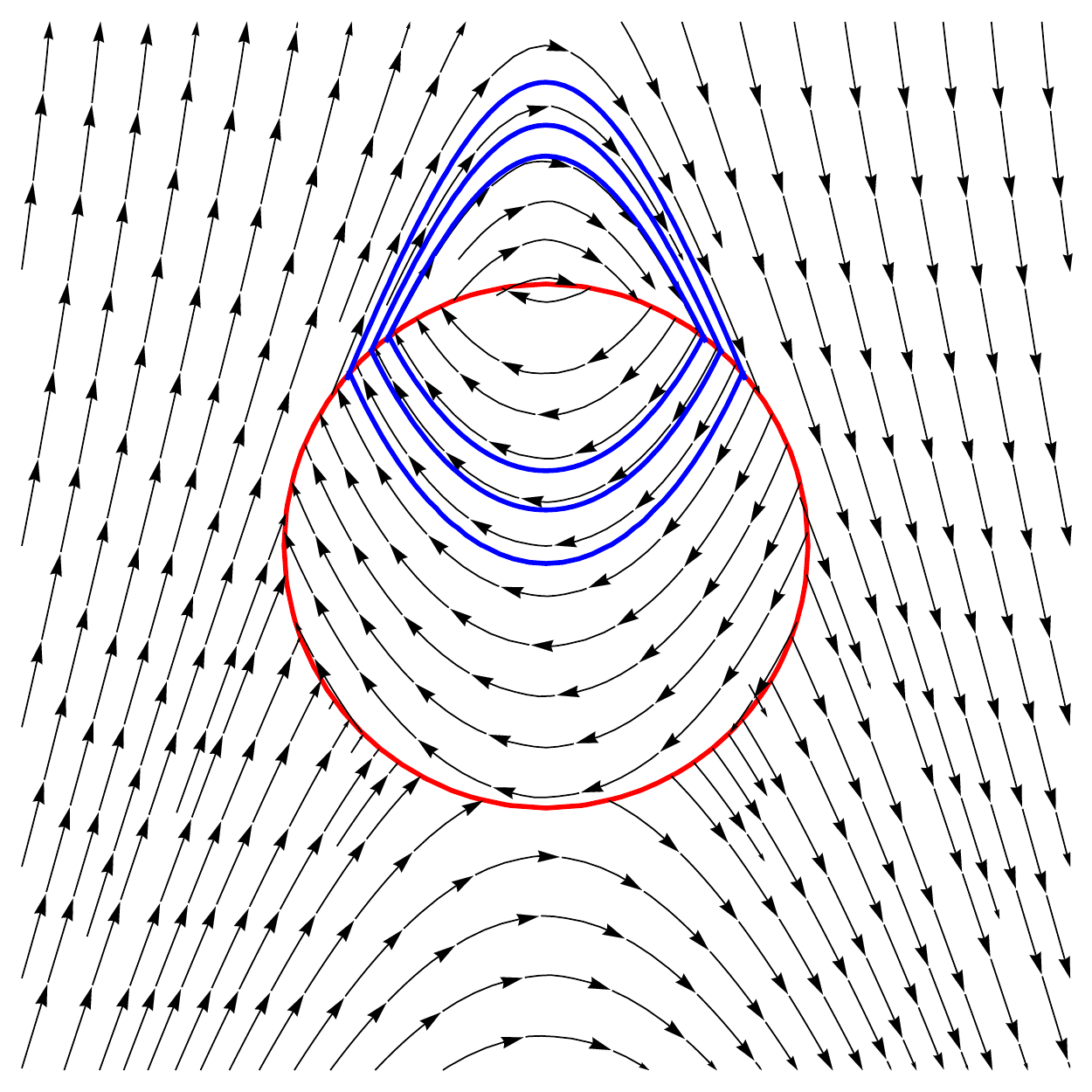}
		\caption{Periodic orbits}
		\label{scconjorb}
	\end{subfigure}%
	\begin{subfigure}{0.33\textwidth}
		\centering
		\includegraphics[scale=0.3]{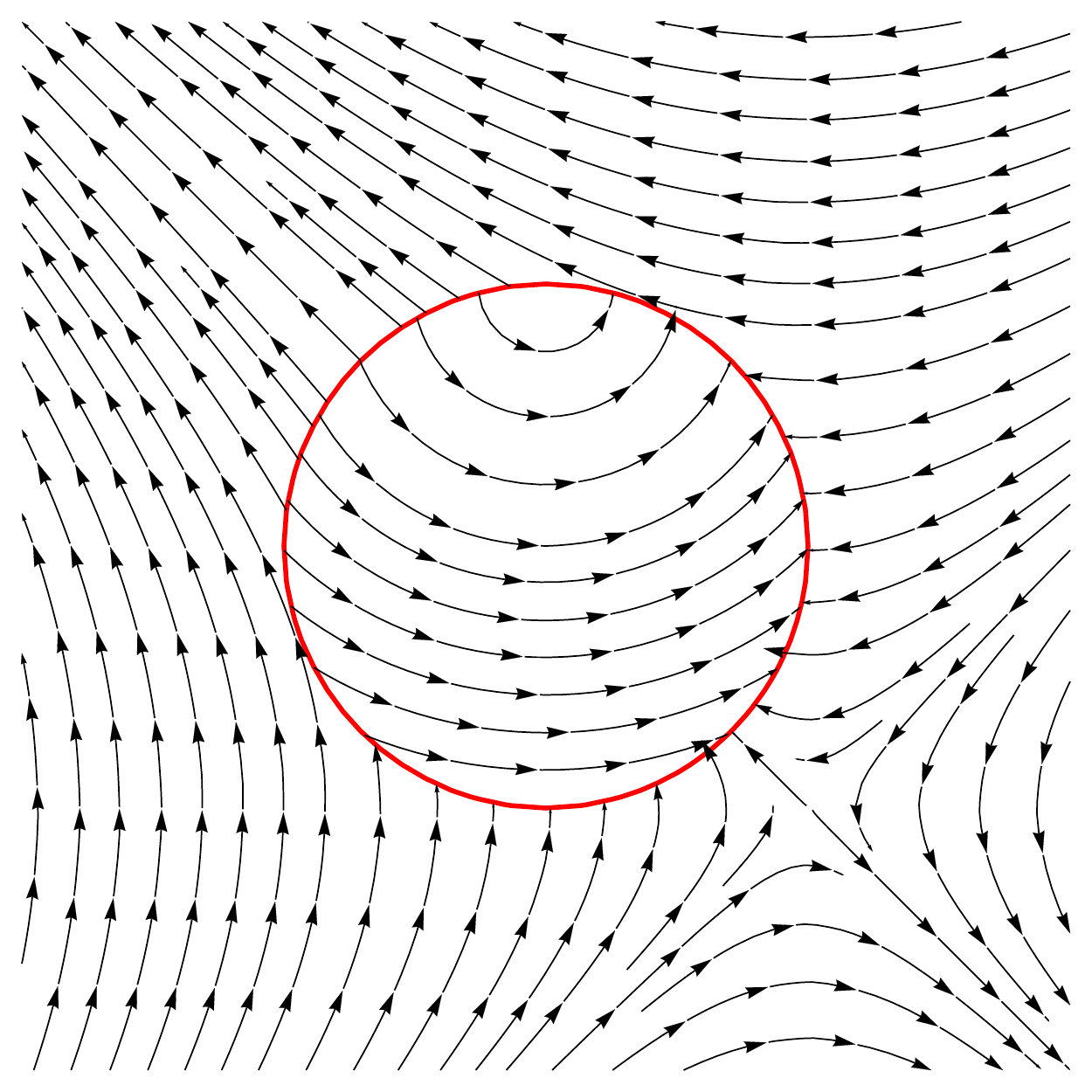}
		\caption{Without periodic orbits}
		\label{centroselasemorb}
	\end{subfigure}%
	\begin{subfigure}{0.33\textwidth}
		\centering
		\includegraphics[scale=0.3]{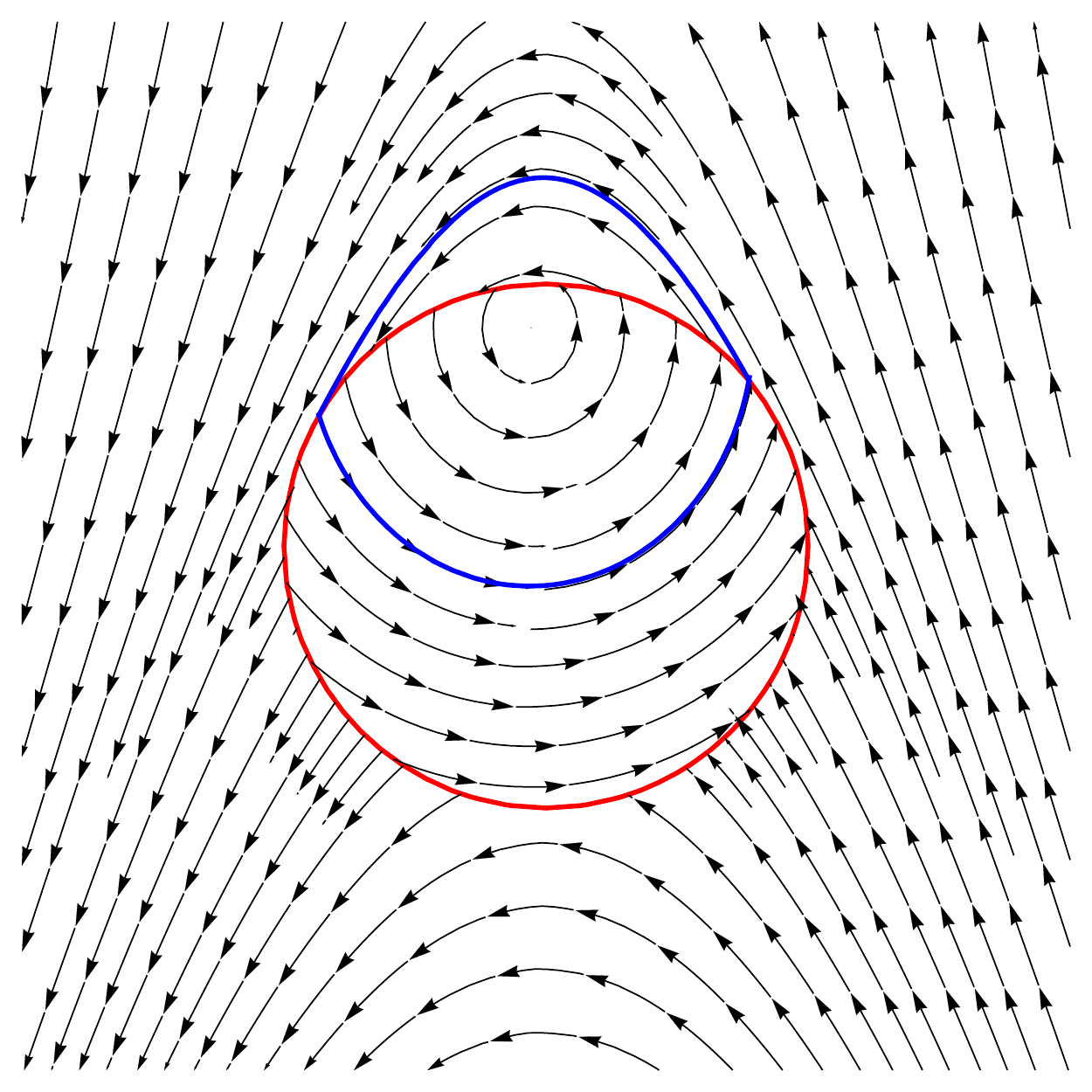}
		\caption{ One limit cycle}
		\label{selacentro1ciclo}
	\end{subfigure}%
	\caption{Phase portrait of vector fields $Z_{XY}\in \mathfrak{X}^{S^1}_3$.}
\end{figure} 

\subsection{Saddle-saddle case}

In this subsection we show that there are vector fields in $\mathfrak{X}^{S^1}_4$ that admit an infinite number of crossing periodic orbits, vector fields that do not admit periodic orbits and vector fields that have only one crossing limit cycle. To do this, we are going to present examples of each case.

Initially, we note that by the Theorem \ref{twocycle} the maximum number of crossing limit cycles that intersect $S^1$ in two points of $Z_{XY}\in \mathfrak{X}^{S^1}_4$ is less than or equal to two, that is, the  Theorem C holds for this case.

\begin{corollary}
	There is $Z_{XY} \in \mathfrak{X}^{S^1}_4$ that admits an infinite number of crossing periodic orbits that intersect $S^1$ in two points. 
\end{corollary}

\begin{proof}
	
	Consider the piecewise smooth vector field $Z_{XY}$, where
	$$X(x,y)=(-6 - 4 y, -10 x)\quad \mbox{and} \quad Y(x,y)=(7 - 2 y, -20 x).$$
	The fields $X$ and $Y$ have the following first integrals 
	$$H_1(x,y)=-6y+5x^2 -2y^2\quad \mbox{and}\quad H_2(x,y)=7y+10x^2 - y^2,$$
	respectively.
	
	To study the existence of crossing limit cycles of the field $Z_{XY}$ we use the closing equations, with which we get the system of nonlinear equations
	$$\left\lbrace \def\arraystretch{1.5} \begin{array}{l}
		-6q + 5p^2 -2q^2=-6s +5r^2- 2s^2,\\
		7q +10p^2 -q^2=7s +10r^2 -s^2,\\
		p^2+q^2=1,\\
		r^2+s^2=1,
	\end{array}\right. $$
	on which $(p,q,r,s)$ satisfies $(p,q)\neq (r,s)$. From the closing equations we have $q^2=1-p^2$ and $s^2=1-r^2$, in this way we obtain the system 
	$$\left\lbrace \def\arraystretch{1.5} \begin{array}{l}
		-6q +5p^2 -2(1-p^2)=-6s +5r^2 -2(1-r^2),\\
		7q +10p^2 -1+p^2=7s +10r^2 -1+s^2,
	\end{array}\right. \Rightarrow 
	\left\lbrace \def\arraystretch{1.5} \begin{array}{l}
		-6q + 7p^2=-6s +7r^2,\\
		7q +11p^2=7s +11r^2,
	\end{array}\right. \Rightarrow$$
	$$ 
	\left\lbrace \def\arraystretch{1.5} \begin{array}{l}
		-42q + 49p^2=-42s +49r^2,\\
		42q +44p^2=42s +44r^2,
	\end{array}\right.  \Rightarrow 93p^2=93r^2 \Rightarrow r=\pm p.$$
	Thus, replacing the acquired information in the first equation of the main system, we get
	$$-6q + 5p^2 -2q^2=-6s + 5r^2 -2s^2 \Rightarrow -6q+5p^2-2(1-p^2)=-6s+5p^2+1-2(1-p^2) \Rightarrow q=s.$$
	Hence, $(p,q,r,s)=(p,q,\pm p,q)$, and since $(p,q)\neq (r,s)$ we cannot have $r=p$. So, $(p,q,r,s)=(p,q,-p, q)$.
	
	Therefore, $Z_{XY}$ admits an infinite number of crossing periodic orbits that intersect $S^1$ in two points, which we illustrate in Figure \ref{ssconjorb}.
	
\end{proof}

\begin{corollary}
	There is $Z_{XY} \in \mathfrak{X}^{S^1}_4$ that does not have crossing periodic orbits that intersect $S^1$ in two points. 
\end{corollary}

\begin{proof}
	Consider the piecewise smooth vector field $Z_{XY}$, on which
	$$X(x,y)=(1 - 2y, -5 - 2x)\quad \mbox{and} \quad Y(x,y)=(2y, 5 + 2x).$$
	The fields $X$ and $Y$ admit the following first integrals
	$$H_1(x,y)=5x +y + x^2 - y^2\quad \mbox{and}\quad H_2(x,y)=-5x - x^2 + y^2,$$
	respectively.
	
	To study the existence of crossing limit cycles of the field $Z_{XY}$ we use the closing equations, with which we get the system of nonlinear equations
	$$\left\lbrace \def\arraystretch{1.5} \begin{array}{l}
		5 p +q + p^2 - q^2 = 5 r + s + r^2 - s^2,\\ 
		-5 p - p^2 + q^2 = -5 r - r^2 + s^2,\\
		p^2+q^2=1,\\
		r^2+s^2=1,
	\end{array}\right. $$
	where $(p,q,r,s)$ satisfies $(p,q)\neq (r,s)$. Adding the first and the second lines above, we determine that $q=s$.
	Thus, since $p^2=1-q^2$ and $r^2=1-s^2$, we have $r^2=p^2$. Replacing this information in the second line of the system, we get
	$$-5 p - p^2 + q^2 = -5 r - r^2 + s^2\Rightarrow -5 p - p^2 + q^2 = -5 r - p^2 + q^2 \Rightarrow p=r.$$
	Then, $(p,q,r,s)=(p,q,p,q)$. 
	
	Therefore, since $(p,q)\neq (r,s)$, $Z_{XY}$ does not admit crossing periodic orbits that intersect $S^1$ in two points, as we illustrate in Figure \ref{sssemorb}.
	
\end{proof}

\begin{corollary}
	There is $Z_{XY} \in \mathfrak{X}^{S^1}_3$ that has exactly one crossing limit cycle that intersects $S^1$ in two points.
\end{corollary}

\begin{proof}
	
	Consider the piecewise smooth vector field $Z_{XY}$, where $$X(x,y)=\left(-\frac{1}{2} + 2x - 4y, -1 - 4x - 2y \right)\quad \mbox{and} \quad Y(x,y)=(5 + 2x - 2y, 1 - 6x - 2y).$$
	The fields $X$ and $Y$ have the following first integrals 
	$$H_1(x,y)=x - \frac{y}{2} + 2x^2 + 2xy - 2y^2 \quad \mbox{and} \quad H_2(x,y)=-x + 5y + 3x^2 + 2xy - y^2,$$
	respectively.
	
	To study the existence of crossing limit cycles of the field $Z_{XY}$ we use the closing equations, with which we get the system of nonlinear equations
	$$\left\lbrace \def\arraystretch{1.5} \begin{array}{l}
		p - \frac{q}{2} + 2p^2 + 2pq - 2q^2=x - \frac{s}{2} + 2r^2 + 2rs - 2s^2,\\
		-p + 5q + 3p^2 + 2pq - q^2=-r + 5s + 3r^2 + 2rs - s^2,\\
		p^2+q^2=1,\\
		r^2+s^2=1,
	\end{array}\right. $$
	on which $(p,q,r,s)$ satisfies $(p,q)\neq (r,s)$. To solve this system we used the Grobner base, obtaining the solutions
	$$\left(-\frac{11 \sqrt{\frac{9067}{137}}}{142}-\frac{18}{71}, \frac{99}{142}-\frac{2 \sqrt{\frac{9067}{137}}}{71}, \frac{11 \sqrt{\frac{9067}{137}}}{142}-\frac{18}{71}, \frac{2 \sqrt{\frac{9067}{137}}}{71}+\frac{99}{142} \right), $$
	$$\left(\frac{11 \sqrt{\frac{9067}{137}}}{142}-\frac{18}{71}, \frac{2 \sqrt{\frac{9067}{137}}}{71}+\frac{99}{142}, -\frac{11 \sqrt{\frac{9067}{137}}}{142}-\frac{18}{71}, \frac{99}{142}-\frac{2 \sqrt{\frac{9067}{137}}}{71} \right), $$
	that represent the same closed curve. 
	
	To complete, we check if these points are in the crossing region. Remember that $h(x,y)=x^2+y^2-1$, then $\nabla h=(2x,2y)$. Thus,
	$$Xh(x,y)=\left\langle  \left(-\frac{1}{2} + 2x - 4y, -1 - 4x - 2y \right), (2x,2y) \right\rangle =-28 x y-12 x-\frac{y}{10},$$
	$$Yh(x,y)=\left\langle  (5 + 2x - 2y, 1 - 6x - 2y), (2x,2y) \right\rangle =14 x-44 x y.$$  
	Hence, 
	$$Xh(p,q)= \frac{12077 \sqrt{\frac{9067}{137}}}{10082}-\frac{9067}{9727}>0,\;\;\; Yh(p,q) =\frac{1175 \sqrt{1242179}-643757}{690617}>0,$$
	$$Xh(r,s) =-\frac{12077 \sqrt{\frac{9067}{137}}}{10082}-\frac{9067}{9727}<0,\;\;\; Yh(r,s) = \frac{-1175 \sqrt{1242179}-643757}{690617}<0,$$
	which imply $Xh(p,q)Yh(p,q)>0$ and $Xh(r,s)Yh(r,s)>0$, that is, $(p,q)$ and $(r,s)$ belong to the crossing region of the discontinuity manifold $S^1$.
	
	Therefore, $Z_{XY}$ has only one crossing limit cycle that intersects $S^1$ in two points, as we illustrate in Figure \ref{ss1ciclo}.
	
\end{proof}

\begin{figure}[h]
	\begin{subfigure}{0.33\textwidth}
		\centering
		\includegraphics[scale=0.3]{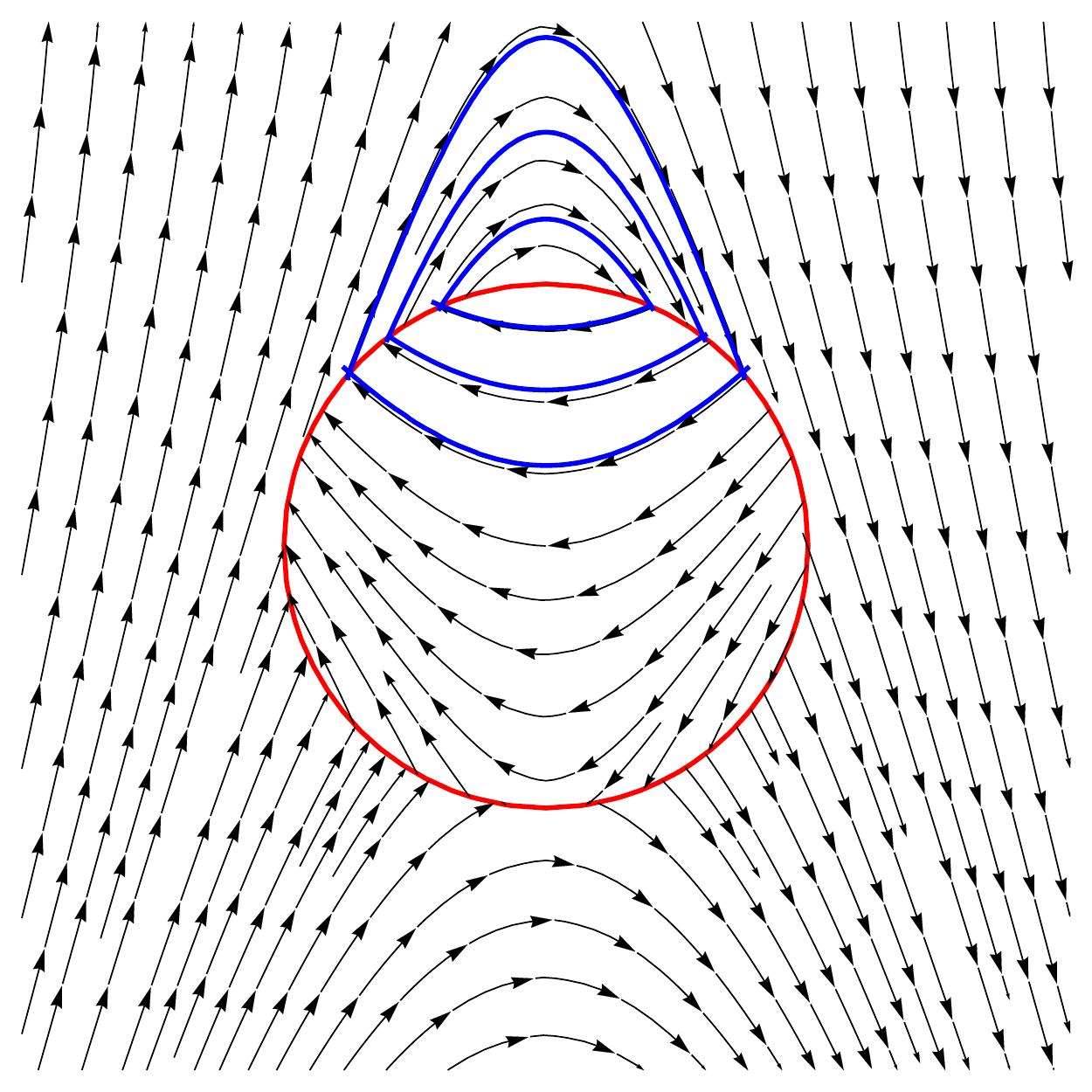}
		\caption{Periodic orbits}
		\label{ssconjorb}
	\end{subfigure}%
	\begin{subfigure}{0.33\textwidth}
		\centering
		\includegraphics[scale=0.3]{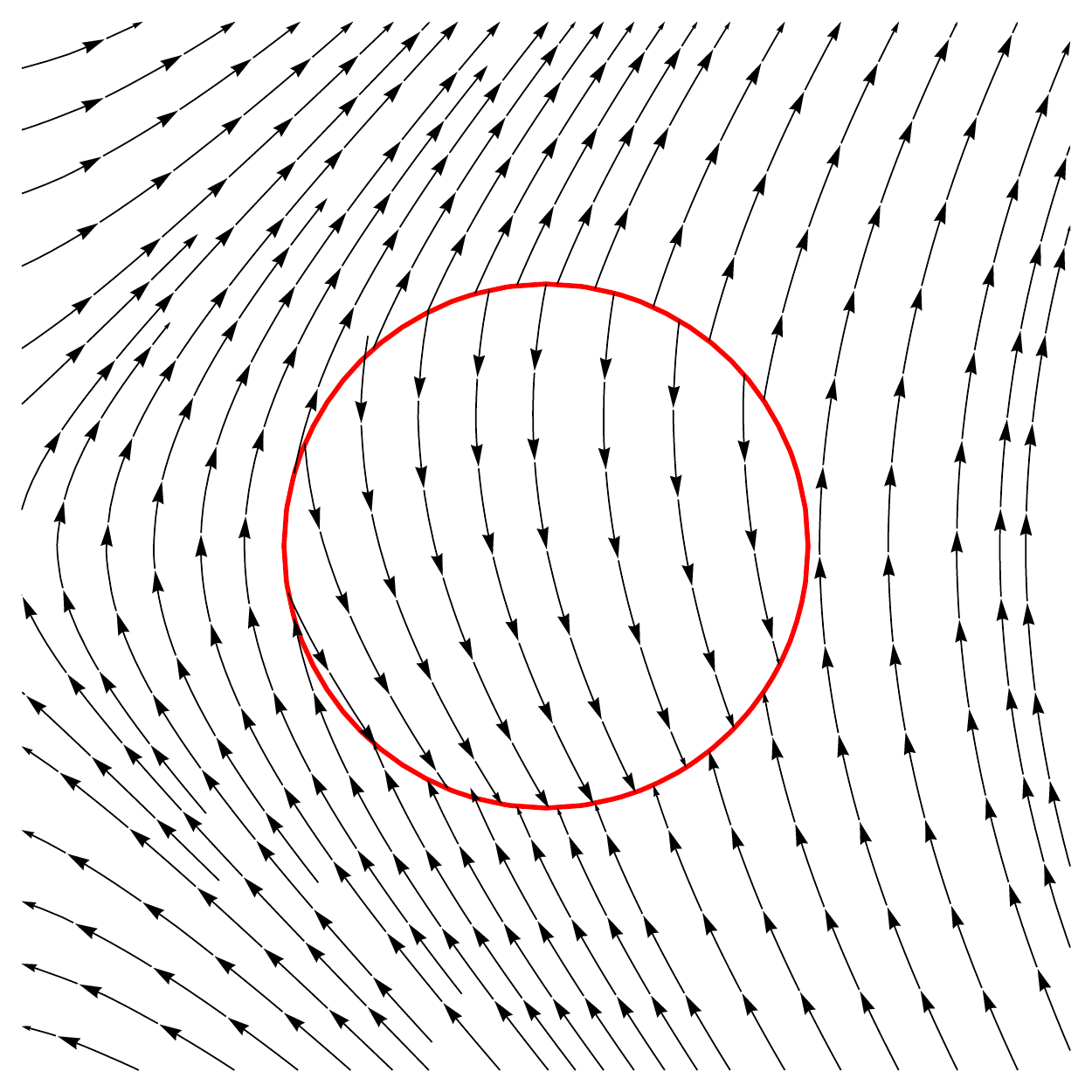}
		\caption{Without periodic orbits}
		\label{sssemorb}
	\end{subfigure}%
	\begin{subfigure}{0.33\textwidth}
		\centering
		\includegraphics[scale=0.3]{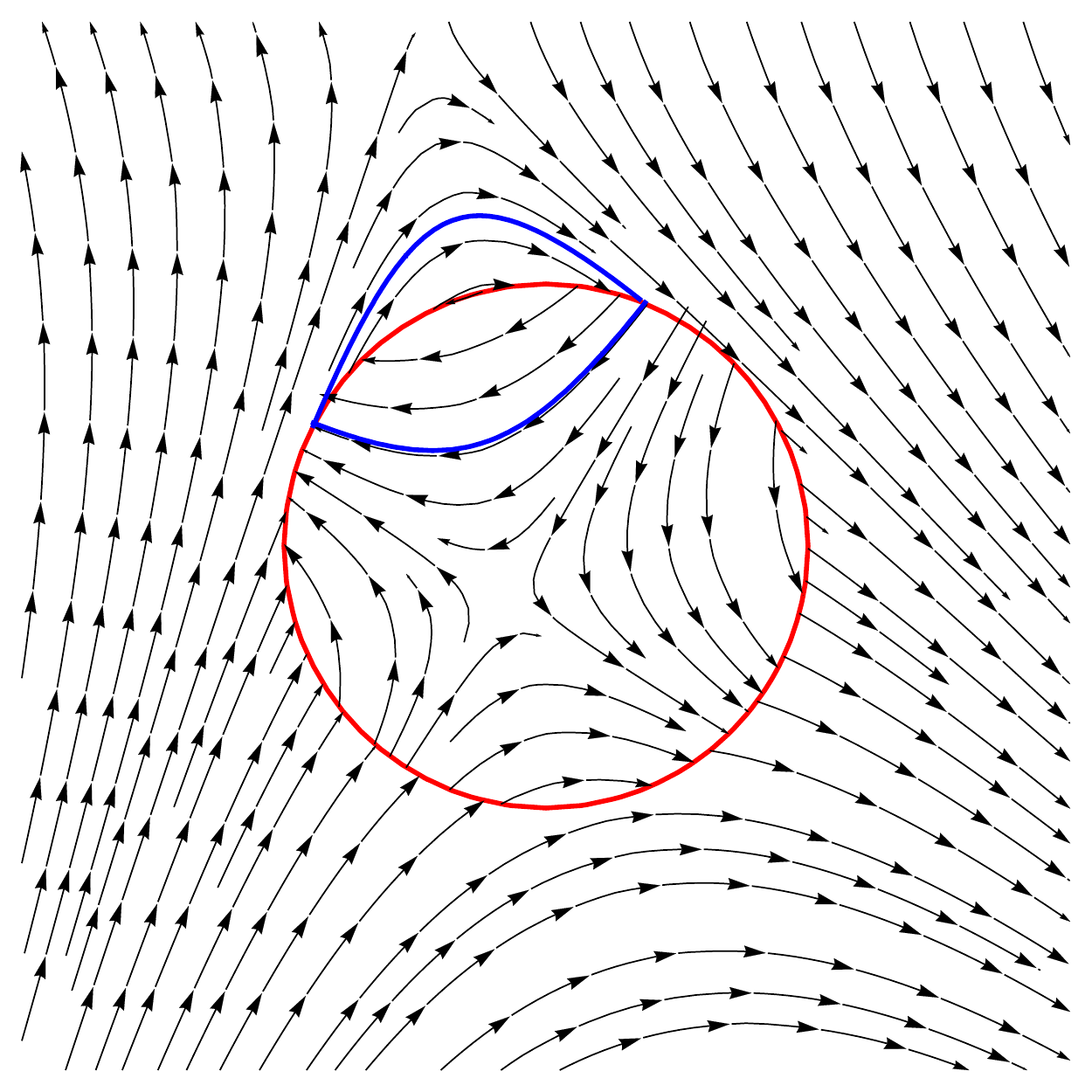}
		\caption{ One limit cycle}
		\label{ss1ciclo}
	\end{subfigure}%
	\caption{Phase portrait of vector fields $Z_{XY}\in \mathfrak{X}^{S^1}_4$.}
\end{figure}

\section*{Acknowledgements}

Mayara Caldas was financed in part by the Coordenação de Aperfeiçoamento de Pessoal de Nível Superior - Brasil (CAPES) - Finance Code 001. Ricardo Martins was partially supported by FAPESP grants 2021/08031-9, 2018/03338--6, 2018/13481--0, CNPq grant 434599/2018--2 and Unicamp/Faepex grant 2475/21. The authors thank Espa\c co da Escrita – Pr\'o-Reitoria de Pesquisa  - UNICAMP - for the language services provided.

The authors declare that they have no conflict of interest.

\end{document}